\documentclass[journal,oneside,web]{ieeecolor}
\usepackage{generic}
\usepackage{cite}
\usepackage{amsmath,amssymb,amsfonts}
\usepackage{algorithmic}
\usepackage{graphicx}
\usepackage{textcomp}
\usepackage{subfigure}
\usepackage{cite}
\usepackage{color}

\usepackage[colorlinks,citecolor=darkblue,urlcolor=red, hyperfigures]{hyperref}
\usepackage{balance}
\usepackage{amsmath}
\usepackage{cases}

\newtheorem{theorem}{Theorem}[section]
\newtheorem{lemma}{Lemma}[section]
\newtheorem{exam}{Example}[section]
\newtheorem{corollary}{Corollary}[section]
\newtheorem{definition}{Definition}[section]

\newtheorem{remark}{Remark}[section]

\newtheorem{assumption}{Assumption}[section]

\def \Re{{\operatorname{Re}}}
\def \col{{\operatorname{col}}}
\def \Ri{{\operatorname{Ri}}}
\def \Rank{{\operatorname{Rank}}}
\def \int{{\operatorname{Int}}}

\definecolor{darkblue}{rgb}{0.0, 0.0, 0.55}
\begin{document}

\title{Distributed Optimal  Output Consensus Control of Heterogeneous Multi-Agent Systems with Safety Constraints}

\author{Ji Ma, Shu Liang, and Yiguang Hong,~\IEEEmembership{Fellow,~IEEE}\vspace{-9pt}
\thanks{Manuscript received \today. (Corresponding author:
Yiguang Hong.)}
\thanks{ J. Ma is with the Department of Automation, Xiamen, Xiamen 361005, China (e-mail: maji08@xmu.edu.cn). }
\thanks{
Shu Liang and Yiguang Hong are with the Department
of Control Science and Engineering, Tongji University, Shanghai, 201210,
China (e-mail: sliang@tongji.edu.cn; yghong@iss.ac.cn).}}
\maketitle

\begin{abstract}
In this paper, we develop a novel dynamic distributed optimal safe consensus protocol to achieve simultaneously safety requirements and output optimal consensus. Specifically, we construct a distributed projection optimization algorithm with an expanding constraint set in the decision-making layer, while we propose a reference-tracking safety controller to ensure that each agent's output remains within a shrinking safety set in the control layer. We also establish the convergence and safety analysis of the closed-loop system using the small-gain theorem and time-varying control barrier function (CBF) theory, respectively.
Besides, unlike previous works on distributed optimal consensus, our approach does not require prior knowledge of the cost function the local objective or gradient function and adopts a mild assumption on the dynamics of multi-agent systems (MASs) by using the transmission zeros condition.
 \end{abstract}

\section{INTRODUCTION}

\IEEEPARstart{I}{n} recent years, distributed optimal consensus has attracted increasing attention due to its extensive applications in various fields such as machine learning, sensor networks, smart grids, and resource allocation \cite{he2018distributed,doan2020distributed,chen2019minimum,li2022survey,dai2019distributed}. The objective of distributed optimal consensus is to design a distributed controller that makes each agent's state or output reach a common optimal point of the sum of convex functions through individual computation and local communication with neighboring agents.
Many distributed optimal consensus algorithms have been developed, such as subgradient algorithms \cite{nedic2009distributed,lou2017privacy}, dual averaging algorithms \cite{wang2010control}, the distributed inexact gradient method, alternating direction methods of multipliers \cite{shi2014linear,liu2019communication}.
For distributed constrained optimization,
 \cite{zeng2016distributed,liu2017convergence} have developed projected algorithms with sublinear convergence rates. Also, \cite{li2021exponentially} have developed a distributed constrained optimization algorithm with an exponential convergence rate.


Many previous works on distributed optimal consensus of multi-agent systems (MASs) do not consider agents' dynamics and simply use the trivial first-order integrator model.
Recently, researchers have begun to deal with various agents' dynamics in distributed control and optimization protocols. For instance, \cite{liu2015second} investigated the distributed optimization problem for MASs with double-integrator dynamics, while  \cite{zhang2015distributed} considered MASs whose dynamics are described by chains of integrators.
The authors in \cite{zhao2017distributed} have designed two distributed adaptive control protocols to address the distributed optimal consensus problem for linear time-invariant MASs.
Also,  \cite{taaang2018optimal,li2019distributed,yu2021new,tang2020optimal}  adopted the ``decision-making and control'' framework, which consists of a distributed optimization algorithm and tracking control law, to solve distributed optimal consensus for linear MASs. Additionally,  \cite{liu2021distributed}  extended this framework to consider the distributed optimal consensus of nonlinear MASs and used the small-gain theorem to establish the convergence.

Safety is an essential component of practical engineering design.
In many control problems, such as multi-robot control and power system control, the controller must drive the system toward an optimal state and avoid unsafe regions.
Most existing works \cite{zhao2017distributed,taaang2018optimal,li2019distributed,yu2021new,tang2020optimal,liu2021distributed} did not consider the safety aspect.
To address the safety problem, two main approaches have emerged: the Barrier Lyapunov Function (BLF) approach \cite{yu2019barrier,tee2009barrier}, and the Control Barrier Function (CBF) \cite{ames2019control,ames2016control}.
The authors in \cite{xu2015robustness}  integrated the Control Lyapunov Function (CLF) and CBF in a unified framework by using quadratic programming. This method effectively addresses the situations where control objectives coincide with safety constraints. In such cases, the closed-loop system prioritizes safety and may sacrifice control objectives to ensure that the system does not violate these safety constraints.

The main challenge to achieve both the optimal consensus and safety for  high-order MAS can be briefly stated as follows.
On the one hand, the optimal solution may lie on the boundary of the satety constraints.
As a result, any strong safety control method that guarantees an interior-point property is inapplicable.
On the other hand, the inertia of the MASs caused by agents' dynamics can make the trajectory run out of the safety zone when it is near the boundary.
We further give some detailed explanations in Example \ref{ex1}.

Motivated by these discussions, we investigate the distributed output optimal consensus control problem of high-order MASs subject to safety constraints.
The main contributions of this work are as follows.


 1) We design a dynamic decision-making and control protocol, which ensures both safety requirements and the optimality, even when the optimal point is on the safety zone's boundary.
To this end, we construct a time-varying expanding set with an adaptive expansion speed. Then, we design a distributed optimization algorithm by using projections with respect to this expanding set and a  reference-tracking controller with a shrinking safety set corresponding to the expanding set.

2) The decision-making design  achieves an input-to-state (ISS) stability to the reference-tracking error, which is a key step for the convergence  of the whole  closed-loop systems.
Also, in contrast to the algorithms with sublinear convergence rates in the literature \cite{zeng2016distributed,liang2017distributed}, our algorithm exhibits an exponential rate.

3) Compared with the existing works \cite{taaang2018optimal,tang2020optimal}, the exact information of the local objective function or gradient function is not required to be known prior. Our methodology enables each agent to utilize only the measured gradient information at its output. In addition, compared with the minimum
phase system assumption in \cite{taaang2018optimal,tang2020optimal} and the rank condition on the dynamics of MASs in \cite{li2019distributed,yu2021new}, we relax them in terms of a transmission zeros condition.

The remainder of the paper is organized as follows: Section \ref{problemstatement} presents the background knowledge and problem statement. Section \ref{design} details the proposed distributed dynamic decision-making and control. Sections \ref{Convergenceanalysis} and \ref{sany} give the convergence and safety analysis of the closed-loop system, respectively. Section \ref{exam} gives a numerical example, followed by the conclusion in Section \ref{con}.

\textit{Notations:}
We use $\mathbb{R}^n$ to denote the Euclidean space of dimension $n$. We write $I_{n}$ for the identity matrix of dimension $n$, $\mathbf{1}_{n}$ for the $n$-dimension column vector with all entries equal to $1$, and $\mathbf{0}_{m\times{n}}$ for the $m\times{n}$ dimension matrix with all entries equal to $0$. We use $\|x\|$ to denote the standard Euclidean norm of a vector $x$ and $\langle\cdot,\cdot\rangle$ to denote the inner product of two vectors. We write $\|W\|$ for the matrix norm induced by the vector norm $\|\cdot\|$. We let $W^{T}$ denote the transpose of a matrix $W$. We use $d(x,y):=\|x-y\|$ to denote the distance between vectors $x$ and $y$. The projection operator of a vector $x$ onto a convex set $\Omega$ is denoted by $P_{\Omega}(x)$, where $P_{\Omega}(x):=\text{argmin}_{v\in\Omega}\|v-x\|^2$. For any two subsets $A,~B\subseteq\mathbb{R}^n$, we denote $A+B:=\{z=x+y\in\mathbb{R}^n|x\in A,~y\in B\}$, $A/B:=\{x\in\mathbb{R}^n|x\in A,~x\notin B\},$ and $\int(A)=A/\partial A.$ A continuous function $\alpha:[0,a)\rightarrow[0,+\infty)$ is said to be a class $\mathcal{K}$ function if $\alpha$ is strictly increasing and $\alpha(0)=0.$  A continuous function $\beta:[0,a)\times[0,+\infty)\rightarrow[0,+\infty)$ is said to be a class $\mathcal{KL}$ function if $\beta(x,y)$ is a  class $\mathcal{K}$ function with respect to $x$, and is decreasing  with respect to $y$ and $\lim_{y\rightarrow+\infty}\beta(x,y)=0.$

\section{Preliminaries and Problem Formulation}\label{problemstatement}
\subsection{Graph theory}
The information exchanged among agents is captured by an undirected and connected graph
$\mathcal{G}=(\mathcal{V},\mathcal{E})$ with $N$ agents, where $\mathcal{V}=\{1,\ldots,N\}$ denotes the vertex set and $\mathcal{E}\subseteq \mathcal{V}\times{\mathcal{V}}$ denotes the edge set. The edge $(i,j)\in \mathcal{E}$ if and only if agents $i$ and $j$ communicate with each other. The graph $\mathcal{G}$ is called connected if there always exists a path
between any two different nodes of the graph. Let $\mathcal{A}=[a_{ij}]_{N\times N}$ be the adjacency matrix of graph $\mathcal{G}$, namely $a_{ij}=1$ if $(i,j)\in\mathcal{E}$, and $a_{ij}=0$, otherwise. We use $D=\text{diag}(d_1,\ldots,d_N)$ to denote the degree matrix with $d_i=\sum_{j=1}^N a_{ij},~\forall i\in{\mathcal{V}}$. The Laplacian matrix of graph $\mathcal{G}$ is denoted by $\mathcal{L}_{\mathcal{G}}=D-\mathcal{A}$. All eigenvalues of $\mathcal{L}_{\mathcal{G}}$ are non-negative and real numbers and can be sorted in ascending order by $0=\lambda_1<\lambda_2\ldots\le\lambda_N$.

\subsection{Control Barrier Function}

We provide the definition for the time-varying Control Barrier Function (CBF):
\begin{definition}\label{def-cbf}
(\cite[Definition 1]{xu2018constrained}) Considering a time-invariant system
\begin{equation}\label{linear sys}
\dot{x}(t)=Ax(t)+Bu(t)
\end{equation}
with the state $x\in\mathbb{R}^n$ and the input $u\in\mathbb{R}^m$, the function $h(x,t)$ is called a time-varying Control Barrier Function (CBF) for system~\eqref{linear sys} if the following condition holds:
\begin{equation}\label{CBF1}
 \sup_{u\in \mathbb{R}^m}\Bigg\{\frac{\partial(h(x,t))}{\partial{x}}(Ax+Bu)+\frac{\partial(h(x,t))}{\partial{t}}+\alpha(h(x,t))\Bigg\}\geq0
\end{equation}
for all $x\in\mathbb{R}^n$ and $t\geq 0$, where $\alpha(\cdot)$ is a class $\mathcal{K}$ function.
\end{definition}
\vspace{0.4em}

The following result follows from \cite[Proposition 1]{xu2018constrained}.
\begin{lemma}\label{leas}
Let $h(x,t)$ be a time-varying CBF for the system \eqref{linear sys} with $h(x(0),0)\geq0$.
Define
\begin{equation}
\begin{aligned}
K(x,t)&=\Big\{u\in\mathbb{R}\Big|\frac{\partial(h(x,t))}{\partial{x}}(Ax+Bu)+\frac{\partial(h(x,t))}{\partial{t}}\\
&\quad\quad+\alpha(h(x,t))\geq0,~\forall t\geq0\Big\}
\end{aligned}
\end{equation}
where $\alpha(s)$ is a class $\mathcal{K}$ function satisfying $\alpha(s)=as,~s\geq0$ for some positive constant $a$. Then $K(x,t)$ is Lipschitz in $x\in \mathbb{R}^n$. Moreover,
the CBF $h(x,t)\geq 0$, for any controller $u\in K(x,t)$ in the system \eqref{linear sys} and all $t\geq0$.
\end{lemma}

\subsection{Problem formulation}
Consider a MAS with $N$ heterogeneous agents, where each agent has a linear dynamics as
\begin{equation}\label{sys}
\begin{aligned}
\dot{x}_i(t)&=A_ix_i(t)+B_iu_i(t)\\
y_i(t)&=C_ix_i(t),\quad i\in\mathcal{V}
\end{aligned}
\end{equation}
where $x_i(t) \in \mathbb{R}^{n_i}$, $u_i(t) \in \mathbb{R}^{m_i}$, and $y_i(t) \in \mathbb{R}^{p}$ represent the state, control input, and output variables of agent $i$, respectively; and $A_i\in \mathbb{R}^{n_i\times n_i}$, $B_i\in \mathbb{R}^{n_i\times m_i}$, and $C_i\in \mathbb{R}^{p\times n_i}$ are the state, input, and output constant matrices, respectively.

The aim of this paper is to design a distributed controller $u_i(t)$ to achieve the following two objectives simultaneously:
\begin{itemize}
  \item[O1:] \emph{\textbf{(Safety Objective)}} The outputs of all agents belong to a preset safety zone $\Omega\subseteq \mathbb{R}^{p}$, i.e.,  $y_i(t)\in\Omega,~t\geq0$.

  \item[O2:]  \emph{\textbf{(Optimal Objective)}} Simultaneously, all agents' outputs $y_i(t),~i\in\mathcal{V},$ converge to an optimal solution $y^*$ to the following distributed optimization problem:
\begin{equation}
\min_{y\in{\Omega}} f(y)=\sum_{i=1}^{N}f_i(y)
\label{problem1}
\end{equation}
where $y\in\Omega\subseteq\mathbb{R}^p$ is an output variable common to all agents,
$f_{i}: \Omega\! \rightarrow\! \mathbb{R}$ is a local objective function private to agent $i$.
\end{itemize}
\begin{remark}
The \emph{safety objective} always stipulates that the outputs of all agents must persistently reside within a predetermined safety set~\cite{wu2021distributed} in practical, but is often ignored in conventional distributed optimal consensus ~\cite{taaang2018optimal,li2019distributed,yu2021new,tang2020optimal}. In fact, when the optimal solution $y^*$ lies at the boundary of the safety zone $\Omega$, the output $y_i(t)$ is likely to rush out of the safe zone $\Omega$ during the process of approaching the optimal point $y^*$. Thereby, it  may lead to a conflict between the \emph{safety objective} and \emph{optimal objective} (see Example.~\ref{ex1}). Addressing this  conflict is the most challenging yet significant aspect of our work. It makes the problem much complicated in comparison with conventional distributed optimal consensus problem considered in the literature, such as \cite{taaang2018optimal,li2019distributed,yu2021new,tang2020optimal}.
\end{remark}

To achieve both \emph{optimal objective} and \emph{safety objective}, we need the following basic assumptions.
\begin{assumption}\label{as1}
The gradient of each local objective function $\nabla f_i(x_i)$ satisfies $\sigma_i$-strongly monotone with some $\sigma_i>0$ and Lipschitz with some constant $L_{i}>0$.
\end{assumption}

\begin{assumption}\label{as2}
The constraint sets $\Omega$  is convex and compact. Moreover, its interior set $\int(\Omega)$ is nonempty.
\end{assumption}

\begin{assumption}\label{as3}
 The surface $\partial\Omega$ has a bounded normal curvature.
\end{assumption}

\begin{assumption}\label{as4}
The pair $(A, B)$ is controllable and the pair $(C, A)$ is detachable.
\end{assumption}

\begin{assumption}\label{as6}
The transmission zeros of each agent's dynamic \eqref{sys} are not equal to zero, i.e.,
\begin{equation}
\Rank\left[
        \begin{array}{cc}
          A_i& B_i\\
          C_i& \bold{0_{p\times m_i}}\\
        \end{array}\label{klsa}
      \right]=n_i+p.
\end{equation}
\end{assumption}
\vspace{0.4em}

\begin{remark}
Assumptions~\ref{as1} and~\ref{as2} guarantee the existence and unique of the optimal solution, which are commonly used in distributed optimization~\cite{zeng2016distributed,liang2017distributed}. Assumptions~\ref{as4} is a standard assumption in distributed output consensus optimization~\cite{an2021distributed}.
 \end{remark}

\begin{remark}
Assumption~\ref{as3} ensures the smoothness of the surface $\partial\Omega$, which can be satisfied in many manifold structures such as spherical and ellipsoidal surfaces. It is worth noting that even when $\partial\Omega$ consists of a piecewise smooth surface with a countable number of corner points (such as a convex polyhedron), our result still remains applicable.
 \end{remark}
\begin{remark}
Assumption~\ref{as6} is weaker than the minimum phase system condition required for each agent's dynamics in~\cite{taaang2018optimal}. It is also weaker than the condition required in~\cite{li2019distributed,yu2021new}, in which the following equality must be satisfied:
  \begin{eqnarray}\label{aasdg}
    \Rank\left[
        \begin{array}{cc}
          -A_iB_i&~~ B_i\\
          C_iB_i&~\bold{0_{p\times m_i}}\\
        \end{array}
      \right]=n_i+p.
\end{eqnarray}
In fact, satisfying~\eqref{aasdg} inherently satisfies Assumption~\ref{as6}, as demonstrated by the following inequality:
\begin{eqnarray}\label{aasdgg}
  &&\Rank \left[
        \begin{array}{cc}
          -A_iB_i&~~ B_i\\
          C_iB_i&~\bold{0_{p\times m_i}}\\
        \end{array}
      \right]\nonumber\\
     &&= \Rank\left(\left[
        \begin{array}{cc}
          A_i&~~ B_i\\
          -C_i&~\bold{0_{p\times m_i}}\\
        \end{array}
      \right]\left[
        \begin{array}{cc}
          -B_i&~~ \bold{0_{n_i\times m_i}}\\
          \bold{0_{m_i\times p}}&~I_p\\
        \end{array}
      \right]\right)\nonumber\\
   &&\leq\Rank\left[
        \begin{array}{cc}
          A_i&~~ B_i\\
          -C_i&~\bold{0_{p\times m_i}}\\
        \end{array}
      \right]=\Rank\left[
        \begin{array}{cc}
          A_i&~~ B_i\\
          C_i&~\bold{0_{p\times m_i}}\\
        \end{array}
      \right].\nonumber
\end{eqnarray}
\end{remark}

\vspace{0.8em}

According to \cite{huang2004nonlinear}, Assumption~\ref{as6} can be replaced by another condition in the following lemma.
\begin{lemma}\label{lem31}
With Assumption \ref{as4}, the following three conclusions hold.
\begin{enumerate}
    \item  Assumption \ref{as6} is satisfied if and only if the  matrix equation
\begin{eqnarray}\label{ch4ass7}
\bold{0_{n_i\times p}}&=&  A_i\Pi_i+B_i\Psi_i\nonumber\\
\bold{0_{p\times p}}&=& C_i\Pi_i-I_{p},   ~~i=1,...,N.
\end{eqnarray}
admits at least one solution pair $(\Pi_i, \Psi_i)$.
\item   If Assumption~\ref{as6} holds, then the matrix $C_i\in \mathbb{R}^{p\times n_i}$ has full row rank.

  \item  When Assumption \ref{as6} holds and  $C_i = [I_p, \boldsymbol{0_{p\times{(n_i-p)}}}]$, the matrix equation \eqref{ch4ass7} yields a particular solution $(\Pi_i, \Psi_i)$ with $\Pi_i=\left[
        \begin{array}{c}
         I_p\\
         \bold{0_{(n_i-p )\times p }}\\
        \end{array}
      \right]$.
\end{enumerate}
\end{lemma}

\begin{proof}
 The proofs of the first and second conclusions can be directly deduced by Theorem 1.9 and Remark 1.11 in \cite{huang2004nonlinear}, and are thus omitted here. Assume that $C_i = [I_p, \boldsymbol{0_{p\times{(n_i-p)}}}]$,  and  $\Pi_i=\left[
        \begin{array}{c}
         I_p\\
         \bold{0_{(n_i-p )\times p }}\\
        \end{array}
      \right]$. In this case, $ C_i\Pi_i-I_{p}=\bold{0_{p\times p}}.$  It follows from \eqref{klsa} that $\Rank(B_i)=n_i.$
Let $A_i=[A_{1,i},A_{2,i}],$ where $A_{1,i}\in\mathbb{R}^{n_i\times p}$ and $A_{2,i}\in\mathbb{R}^{n_i\times (n_i-p)}$.
 Since $\Rank(B_i)=n_i$, there exists a matrix $ \Psi_i$ such that $-A_{1,i}=B_i\Psi_i$ holds for any $A_{1,i}\in\mathbb{R}^{n_i\times p}$. Consequently,  the equation $A_i\Pi_i+B_i\Psi_i=\bold{0_{n_i\times p}}$ is valid, confirming the third conclusion.
\end{proof}

Since the matrix $C_i\in \mathbb{R}^{p\times n_i}$ has full row rank, without loss of generality, we denote $C_i = [I_p, \boldsymbol{0_{p\times{(n_i-p)}}}]$. Then, we propose the following assumption, which is required in our safety analysis.
\begin{assumption}\label{as42}
The initial state of each agent $i$ satisfies that $x_{ij}(0)=0$ for any $j\in (p,n_i]$, where $x_{ij}(0)$ denotes the $j$-th element of $x_{i}(0)$.
\end{assumption}

\begin{remark}
Assumption~\ref{as42} can easily be satisfied in most distributed applications, such as multi-robot systems and multi-drone systems, in which the initial position of each agent is generally located at a stationary state within the constraint set $\Omega$. In addition, for distributed state consensus optimization (in this case, $x_i=y_i$), we do not need Assumption~\ref{as42} any more.
\end{remark}

 \subsection{Detailed explanations of the challenge}
 Unlike the existing works on first-order integrator MASs \cite{zeng2016distributed,liang2017distributed,li2021exponentially}, we consider the high-order dynamic characteristics of the MASs in  \eqref{sys} and adapt the decision-making and control framework presented in \cite{taaang2018optimal,li2019distributed,yu2021new,tang2020optimal} accordingly.

However, when we combine this framework with the traditional constraint optimization methods,
such as the traditional projection method  \cite{zeng2016distributed,liang2017distributed} or the interior point method \cite{li2021exponentially} to deal with the distributed optimal output consensus control problem with safety constraints, the main challenge lies in the  conflict between the \emph{optimal objective} and the \emph{safety objective}. This conflict is inevitable, especially when the optimal solution $y^*$ is at the boundary of the constraint set $\Omega.$ Therefore, we cannot directly integrate the existing distributed optimization constraint algorithms \cite{zeng2016distributed,liang2017distributed,li2021exponentially} with safety control approaches \cite{yu2019barrier,tee2009barrier,ames2019control,ames2016control} into the decision-making and control framework \cite{zhao2017distributed,li2019distributed,yu2021new}.
To further elucidate this  conflict, we present an example as follows.

\begin{exam}\label{ex1}
Consider a source-searching problem involving two unmanned aerial vehicles (UAVs). The local cost functions for both vehicles are defined as: $f_1(y)=f_2(y)=\frac{1}{2}\|y-y^*\|^2$, where $y^*$ denotes the location of the source. Each UAV is constrained to operate within a safety zone $\Omega$.

We employ the decision-making and control framework \cite{zhao2017distributed,li2019distributed,yu2021new} to address this source-searching problem. Specifically, the decision-making layer utilizes a distributed optimization algorithm constrained by the set
$\Omega$ (e.g., the distributed continuous-time projected algorithms in \cite{zeng2016distributed,liang2017distributed} or distributed optimization algorithm based on interior point method \cite{li2021exponentially}) to generate a reference signal
$s_i(t)$ aimed at identifying the optimal source location
$y^*$. In the control layer, techniques such as CBF or BLF are applied to design a reference-tracking controller
$u_i$ corresponding to the system dynamics described by equation \eqref{sys}. This ensures that the output
$y_i(t)$ not only tracks the reference signal $s_i(t)$
but also remains within a sufficiently small neighborhood of $s_i(t)$.

\begin{figure}
\xdef\xfigwd{\textwidth}
    \centering
\includegraphics[width=0.45\linewidth]{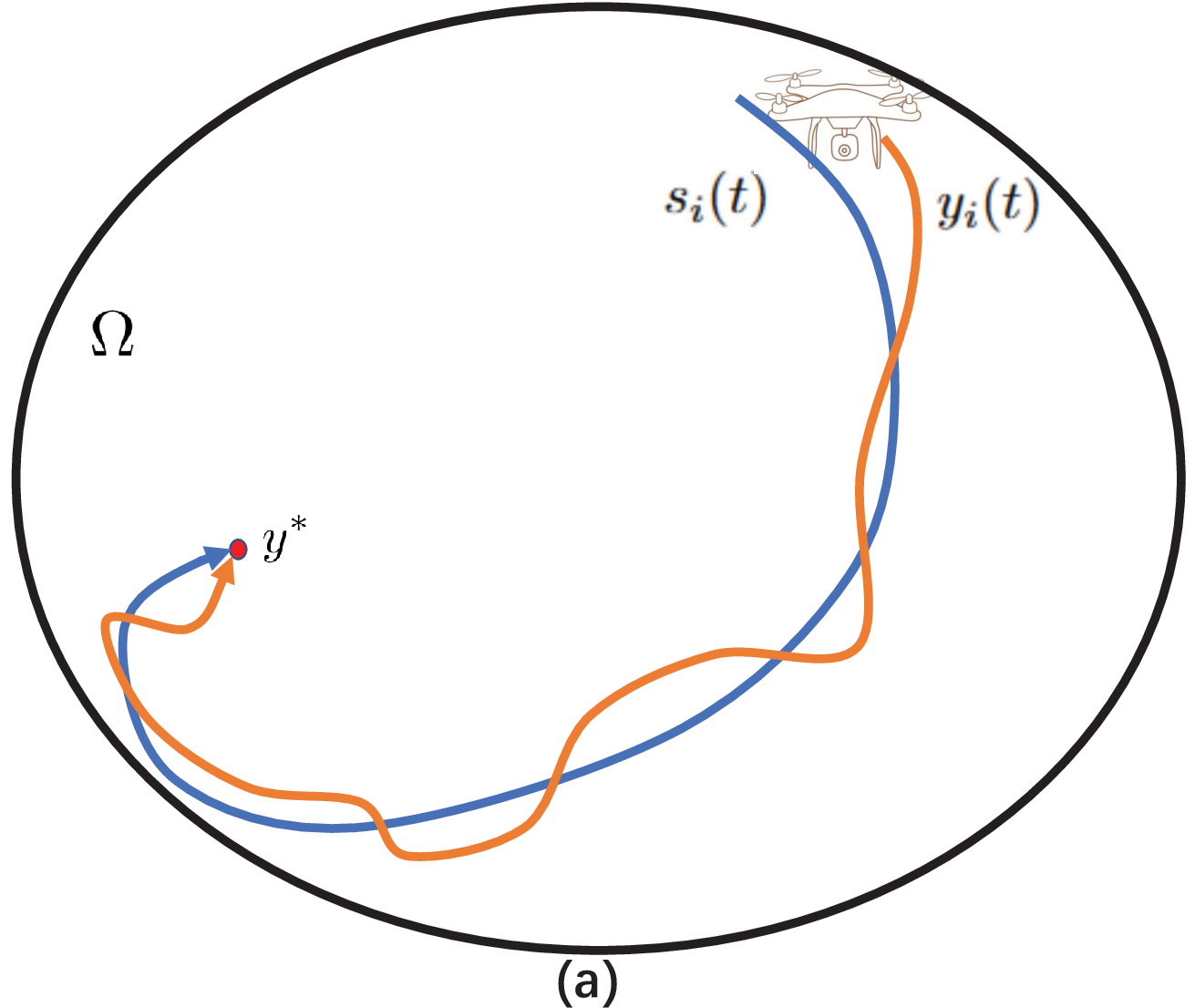}
    \label{1a}\hfill
        \includegraphics[width=0.45\linewidth]{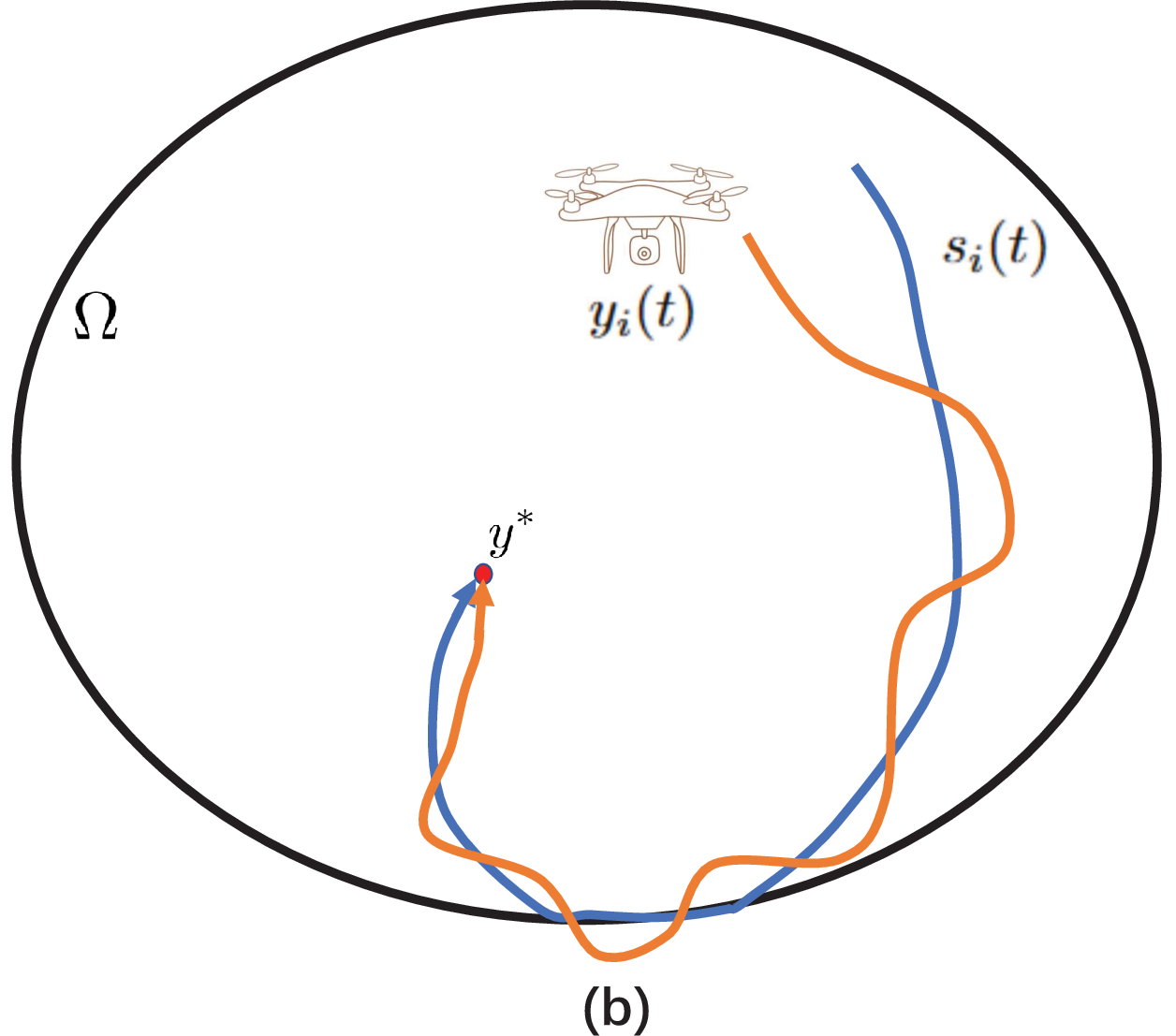}
    \label{1b}\\
        \includegraphics[width=0.45\linewidth]{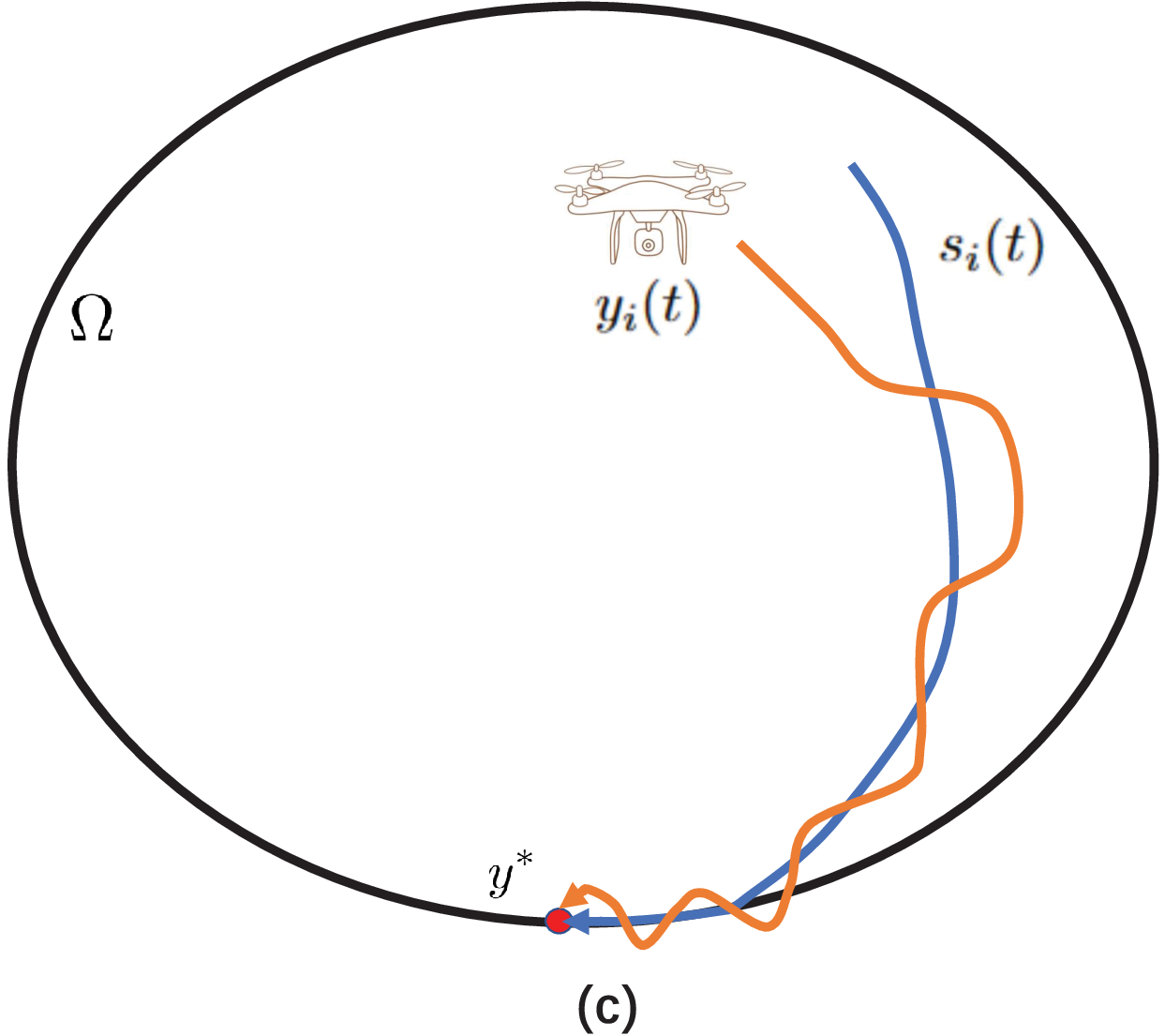}
    \label{1c}\hfill
        \includegraphics[width=0.45\linewidth]{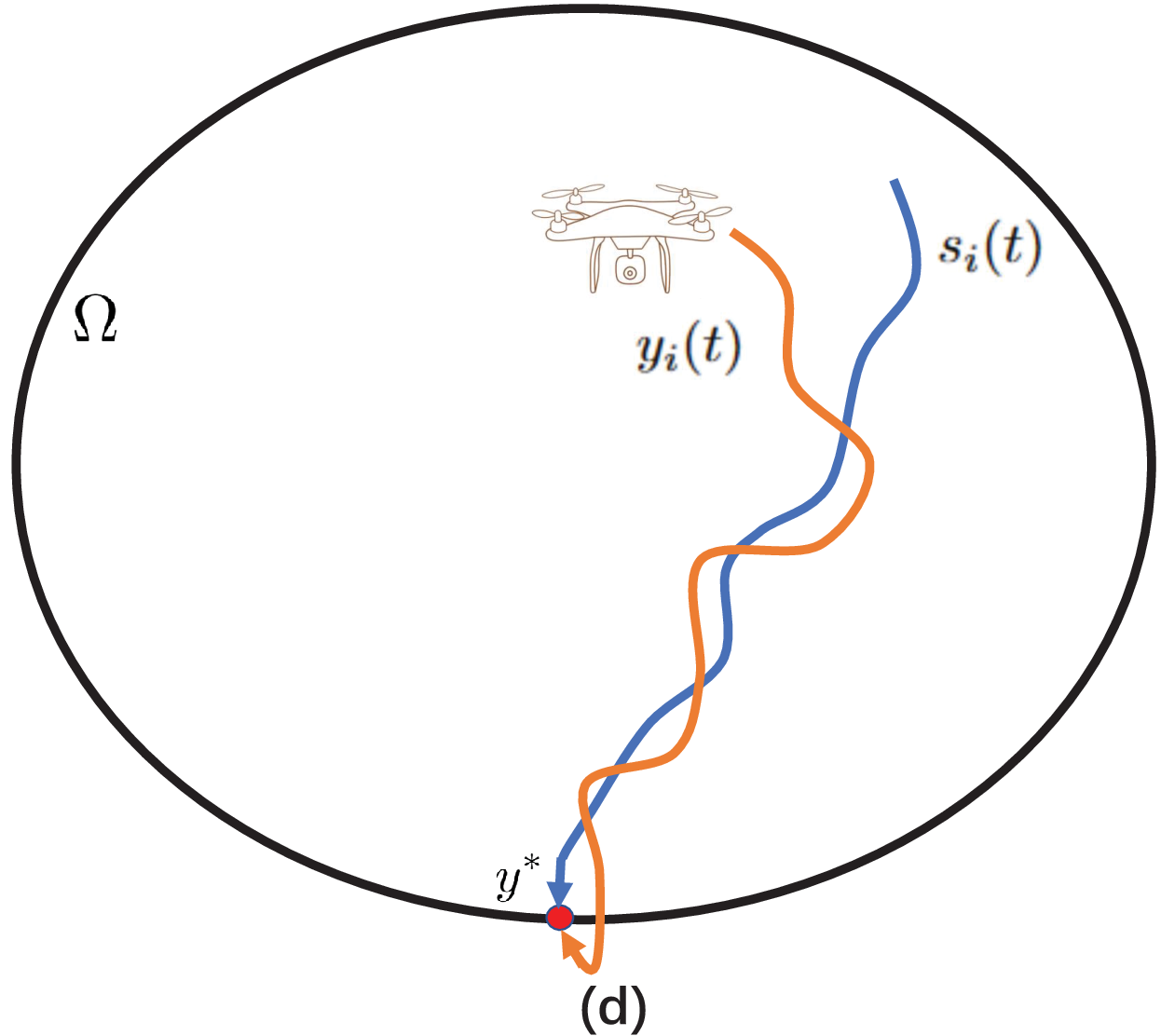}
     \label{1d}
	  \caption{Four possible trajectories of reference signals $s_i(t)$ and outputs $y_i(t)$. The blue line depicts the trajectory of the reference signal $s_i(t)$, the orange line shows the output trajectory $y_i(t)$, and the solid circle represents the constrained safety zone $\Omega$.}
	  \label{fig1}
\end{figure}

In the aforementioned implementation process, only four distinct trajectories for the reference signal $s_i(t)$ and the output $y_i(t)$
 are possible, as illustrated in Fig.\ref{fig1}. Specifically, Fig.~\ref{fig1}(a) shows that when the reference signal trajectory $s_i(t)\in\int(\Omega)$, the trajectory of output $y_{i}$ can also be effectively constrained in $\Omega$, given that $y_{i}(t)$ consistently remains within a sufficiently small neighborhood of $s_{i}(t)$. In this case, when any controller is employed, there is a risk of the output $y_{i}(t)$ being pulled out of the safety zone, as illustrated in Fig.~\ref{fig1}(b)-Fig.~\ref{fig1}(d), implying the invalid the \emph{safety objective}.

\begin{figure}
\xdef\xfigwd{\textwidth}
    \centering
       \includegraphics[width=0.45\linewidth]{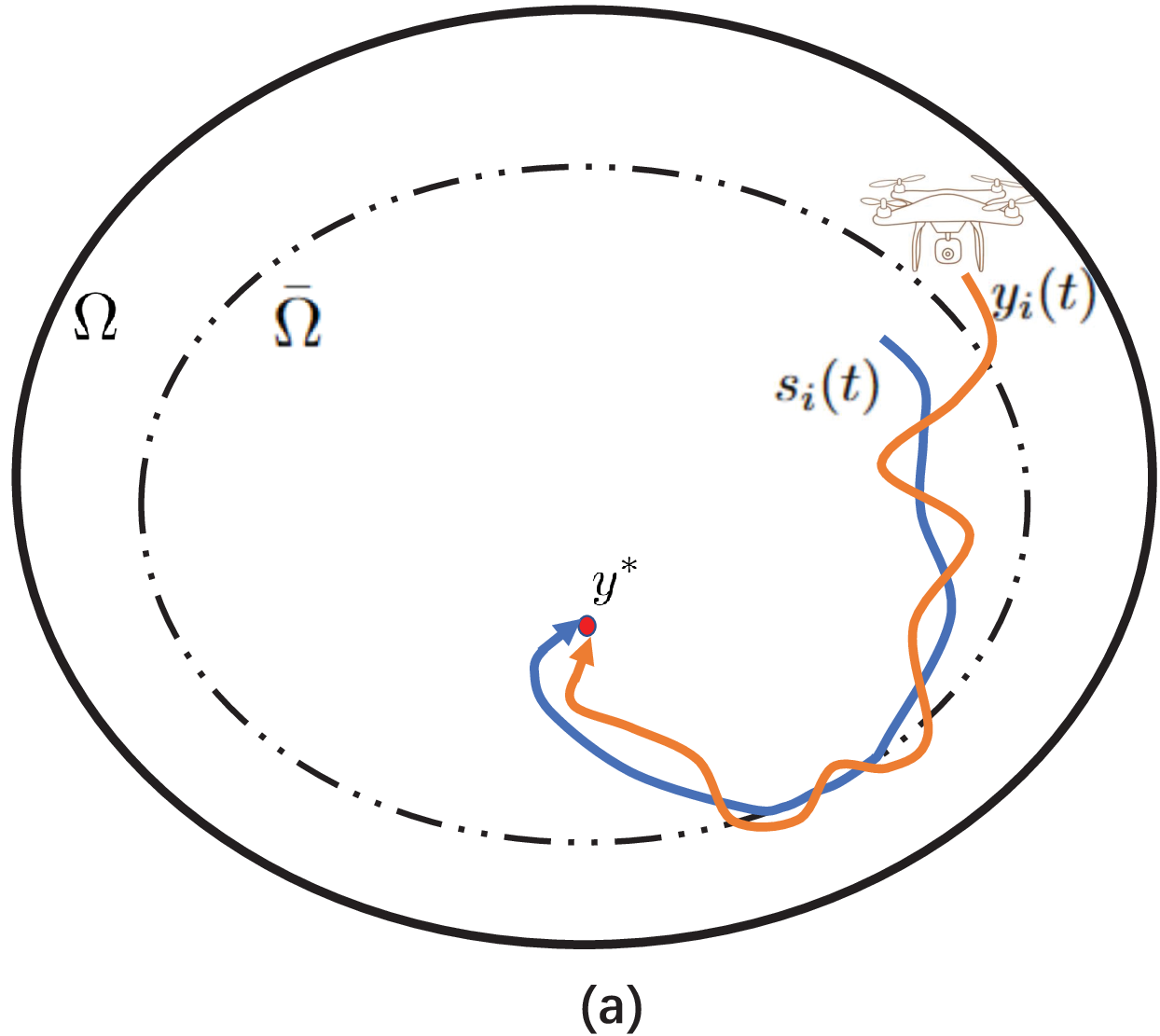}
    \label{1a}\hfill
        \includegraphics[width=0.45\linewidth]{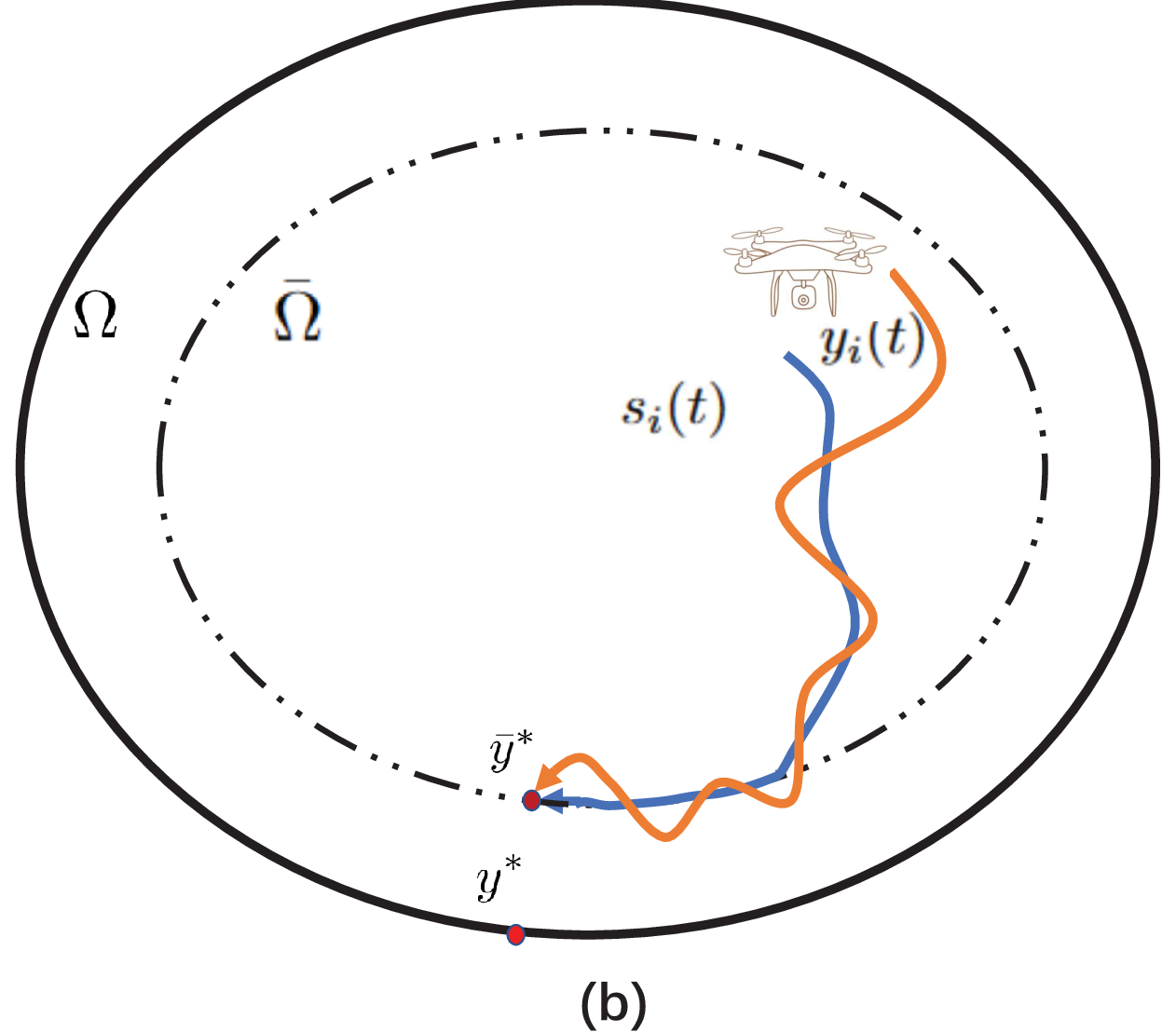}
    \label{1b}
	  \caption{Two possible trajectories of reference signals $s_i(t)$ and outputs $y_i(t)$. The blue line depicts the trajectory of the reference signal $s_i(t)$, the orange line shows the output trajectory $y_i(t)$, the solid circle represents the constrained safety zone $\Omega$, and the dashed circle indicates an auxiliary safety zone $\bar{\Omega}$.}
	  \label{fig12}
\end{figure}

Conversely, prioritizing safety may compromise optimization accuracy. To address this, we propose an approach that employs an auxiliary constraint set $\bar{\Omega}\subseteq \Omega$ to ensure the output $y_{i}$ consistently remains within the original safety zone $\Omega$, as illustrated in Fig.~\ref{fig12}. When source location (i.e., the optimal solution of~\eqref{problem1}) $y^*$ is within the ${\Omega}$, both safety and optimality can be achieved simultaneously, as depicted in Fig.~\ref{fig12}(a). However, when $y^*$ lies on the boundary of $\Omega$, the reference signal $s_{i}(t)$, constrained by $\bar{\Omega}$, converges only to a suboptimal point $\bar{y}^*$. Consequently, the output $y_{i}$ also converges to $\bar{y}^*$, thus sacrificing the \emph{optimal objective}.
\end{exam}

Based on Example \ref{ex1}, integrating the existing distributed optimization constraint algorithms with safety control approaches may lead  to a  conflict  between optimal and safety objectives. It indicates that the traditional projection method \cite{zeng2016distributed,liang2017distributed} or interior point method \cite{li2021exponentially} is unsuitable for our problem. To address this, we propose an adaptive expanding set projection method and design a dynamic decision-making and control protocol in the next section.

\section{Distributed dynamic decision-making and control protocol based on expanding set  projection}\label{design}
For \eqref{sys},  we propose the following distributed dynamic decision-making and control protocol using the projection on an expanding set $\bar\Omega(t)$,  which can both  achieve  \emph{safety} and  \emph{optimal} objectives.
\begin{subequations}\label{24e1}
\begin{numcases}{}
\dot z_i(t)=\frac{k_2}{\beta}\sum_{j=1}^Na_{ij}(z_j(t)-z_i(t))+\frac{d\nabla f_{i}(y_{i}(t))}{dt}\label{24e11}\\
\dot {\eta}_i(t)=\beta k_1\sum_{j=1}^Na_{ij}(\eta_{j}(t)-\eta_{i}(t))\nonumber\\
\quad\quad\quad\quad+\beta\left(s_i(t)-\eta_{i}(t)\right)-\alpha\beta z_i(t)\label{24e122}\\
s_i(t)=P_{\bar{\Omega}(t)}(\eta_i(t))\label{24e133}\\
u_{i}(x_i,s_i)=K_{1i}x_i(t)+K_{2i}s_i(t) \label{24e12}
\end{numcases}
\end{subequations}
where  $y_{i}(t)\in\mathbb{R}^p$ is the output of agent $i$, $z_i(t)\in\mathbb{R}^p$ and $\eta_i(t)\in\mathbb{R}^p$ are the intermediate variables, $s_i(t)\in\mathbb{R}^p$ denotes a decision variable, and $u_{i}(x_i,s_i)$ is the controller.  Note that the initial value of $\eta_{i}(0)\in \int (\Omega)$ is arbitrary, and $z_{i}(0)$ satisfies $z_{i}(0)=\nabla f_{i}(y_{i}(0))$.

The parameters of the distributed optimal control protocol are designed as follows.

\begin{itemize}
    \item The gain matrix $K_{1i}$ is a constant matrix satisfying that  $A_i+BK_{1i}$ is Hurwize, and
the matrix $K_{2i}$ is given by $K_{2i}=\Psi_i-K_{1i}\Pi_i$. Note that the pair $(\Pi_i, \Psi_i)$ is a particular solution to matrix equation \eqref{ch4ass7} satisfying $\Pi_i=\left[
        \begin{array}{c}
         I_p\\
         \bold{0_{(n_i-p )\times p }}\\
        \end{array}
      \right]$, whose existence is validated in Lemma \ref{lem31}.
\item The parameter $\alpha$ can be chosen as any positive constant, and $\beta$ is a positive constant not greater than $1$.

\item The parameters $k_1$ and $k_2$ are sufficiently large such that
\begin{eqnarray}
k_1\lambda_2-3\geq N\beta_2(\alpha)/2
\end{eqnarray}
and
\begin{align}
k_2\lambda_2 - \frac{72\alpha^2L^2_{\max}}{k_2} &\geq \frac{2\beta\alpha^2}{k_1\lambda_2} \nonumber \\
\frac{8L^2_{\max}}{k_2}\left(6\lambda_N^2 k_1^2 + 24 + 18\alpha^2L^2_{\max}\right) &\leq k_1\lambda_2 - 3 \nonumber \\
\frac{8L^2_{\max}N}{k_2}\left(6\lambda_N^2 k_1^2 + 24 + 18\alpha^2L^2_{\max}\right) &\leq \frac{\beta_1(\alpha)}{2}
\end{align}
where
\begin{align}
   \beta_1(\alpha) &:= \min\Big\{\frac{\sigma\alpha}{4} - \alpha^2\Big(L_{\text{max}} + \|\nabla f(y^*(t))\|M_1\Big)^2 \nonumber \\
   &~~~~ -3\alpha^2L_{\text{max}}^2,~\frac{1}{6} - 2\sigma\alpha, \frac{1}{2}(\sigma_1 - 2\alpha^2L_{\text{max}}^2), \frac{1}{4}\Big\} \nonumber \\
   \beta_2(\alpha) &:= \max\Big\{\frac{6\alpha L^2_{\text{max}} + 3}{N\sigma} + \frac{2 + 4\alpha^2L_{\text{max}}^2}{N}, \nonumber \\
   &~~~~~~~~~~~~~~~~~~~~~~~~~~~~~~~~\frac{6}{\alpha^2L^2_{\text{max}}N^2} + \frac{12}{N^2}\Big\} \nonumber \\
    \beta_3(\alpha) &:= \max\Big\{\frac{6\alpha L^2_{\text{max}}}{\sigma N} + \frac{4\alpha^2L^2_{\text{max}}}{N},\frac{12}{N^2}\Big\} \nonumber \\
    \beta_4(\alpha) &:= \max\Big\{\frac{3}{\sigma\alpha} + 6, \frac{6}{\alpha^2L^2_{\text{max}}N} + 4\Big\}
\end{align}
 $L_{max}:=\max_{i\in\mathcal{V}}L_i$, $\sigma:=\sum_{i=1}^N\sigma_i/N,$ $\sigma_1:=\min\{\sigma,\frac{1}{16}\},$ and $M_1$ is an upper bound for the  normal curvature of $\partial \bar\Omega(t).$

\end{itemize}

The constraint set $\bar\Omega(t)$ is an expanding convex set and is required to satisfy the following conditions:
\begin{itemize}

\item[\textbf{C1:}]\label{c1} $\bar{\Omega}(t)\subseteq \int(\Omega)$,~ $\bar{\Omega}(t_1)\subseteq\bar{\Omega}(t_2)$, $\forall~t_1\leq t_2;$

\item[\textbf{C2:}] $\lim_{t\rightarrow\infty}\bar{\Omega}(t)=\Omega$;

\item[\textbf{C3:}] The surface $\partial\bar\Omega(t)$ has a bounded normal curvature.

\end{itemize}

\begin{remark}
 Note that C1 presents the expanding property of $\bar{\Omega}(t)$ and ~C2 illustrates that this set $\bar{\Omega}(t)$ approximates the actual safety zone $\Omega$ when time tends to infinity. In addition, Assumption~\ref{as2} ensures  $\Ri\big(\Omega\big)\neq \emptyset$, which further ensures the existence of  $\bar{\Omega}(t)$ satisfying C1 and C2. From Assumption~\ref{as3},  there must exist an $\bar{\Omega}(t)$ such that $\partial\bar\Omega(t)$ consistently exhibits bounded normal curvature, i.e., C3 holds.
\end{remark}

\begin{remark}
  In dynamic~\eqref{24e11}, agent $i$ only uses the gradient information $\nabla f_{i}(y_{i})$, which can be measured in practical distributed optimization applications. For example, in the source-seeking problem of multi-robot systems, each robot can only measure the gradient value of an unknown
signal field. In this case, the gradient information $\nabla f_{i}(s_{i})$ is available.
However, our strategy will lead to an error between $\nabla f_{i}(y_{i})$ and $\nabla f_{i}(s_{i})$, which further compromises the tracking accuracy of $s_{i}(t)$ to $y^*(t)$.
This issue is  often faced by existing distributed output optimization algorithms that are sensitive to perturbations (see, e.g.,~\cite{zeng2016distributed,liang2017distributed}).
Hence, in the design of the dynamic decision-making and control protocol \eqref{24e1}, the decision-making algorithm \eqref{24e11}-\eqref{24e133} is developed by a new distributed optimization algorithm \eqref{2ss4e1}, which exhibits robustness to external disturbances.
In our theoretical convergence analysis (Sec.~\ref{Convergenceanalysis}), we solve this issue and prove that $s_{i}(t)$ exactly converges to the time-varying optimal solution $y^*(t)$ to the following time-varying optimization problem:
\begin{equation}
\min_{y \in \bar{\Omega}(t)} f(y) = \sum_{i=1}^{N} f_i(y).
\label{problem3}
\end{equation}

\end{remark}

The  dynamic decision-making and control protocol is designed based on an adaptive expanding set projection method. To illustrate this method, we construct a \emph{shrinking safety set}  $\mathcal{S}_{s_i,\bar\Omega}(t)$ corresponding to the \emph{expanding constraint set} $\bar\Omega(t)$. The \emph{shrinking safety set} $\mathcal{S}_{s_i,\bar\Omega}(t)$ is a $p$-dimension open ball $\mathcal{B}(s_i(t), r(t))$ with $s_i(t)$ as the center and $r(t)$ as the radius, where its radius satisfies
\begin{equation}\label{uhas}
r(t)=\min\big\{r\geq0 \big|x+y\in\Omega,~\forall x\in\bar\Omega(t),~y\in \mathcal{B}(0,r)\big\}.
\end{equation}

\begin{remark}\label{r3}
Since $\bar\Omega(t)$ is expending over  time, \eqref{uhas} implies $\lim_{t\rightarrow\infty}r(t)=0$, which further ensures the contracting property of $\mathcal{S}_{s_i,\bar\Omega}(t)$. Moreover, \eqref{uhas} indicates that $\mathcal{S}_{s_i,\bar\Omega}(t)$ always remains within the original safety zone $\Omega$.
\end{remark}

Based on the design of the \emph{expanding constraint set} $\bar\Omega(t)$ and the corresponding \emph{shrinking safety set}  $\mathcal{S}_{s_i,\bar\Omega}(t)$, our protocol uses an adaptive expanding set projection, as shown in Fig~\ref{figg1a}.
This method can effectively separate the designs of distributed decision-making layer and control layer as follows.
\begin{itemize}
    \item[(i)] In the decision-making layer, design a  distributed projection optimization algorithm \eqref{24e11}-\eqref{24e133} with the \emph{expanding constraint set} $\bar{\Omega}(t)$ to solve the distributed optimization problem \eqref{problem3} .

    \item[(ii)] In the control layer, design a safety reference controller \eqref{24e12} so that the output $y_i(t)$ can track the signal $s_i(t)$. Choose a suitable \emph{expanding constraint set} $\bar\Omega(t)$ with a sufficiently slow expansion speed such that \eqref{disf} and  \eqref{aasfdds} hold. Correspondingly, the \emph{shrinking safety set} $\mathcal{S}_{s_i, \bar{\Omega}}(t)$ has slow enough shrink speed such that the output $y_i$ remains within $\mathcal{S}_{s_i, \bar{\Omega}}(t)$.
   \end{itemize}

According to the second point in the extended set projection method, the expansion speed of \emph{expanding constraint set} is adaptively determined by the trajectory of the output $y_i$, and is also determined by the closed-loop system composed of the MASs~\eqref{sys} and the distributed protocol \eqref{24e1}.
This expansion speed is critical for achieving both the \emph{optimal objective} and the \emph{safety objective}, particularly when the optimal point $y^*$ is on the boundary of the
safety zone $\Omega$. The decision signal
$s_i(t)$ is constrained within the \emph{expanding constraint set} $\bar\Omega(t)$, where the expansion rate determines the speed at which the signal $s_i(t)$ approaches the set boundary $\partial\Omega.$ It can prevent $s_i(t)$ from approaching the constraint boundary ``too quickly'', thereby it prevents the agent's output $y_i(t)$ from exiting the safety zone $\Omega$, as depicted in Fig.\ref{fig1}(c)-Fig.\ref{fig1}(d). Consequently, the proposed distributed dynamic decision-making and control protocols \eqref{24e1} can achieve the \emph{optimal objective} without violating the \emph{safety objective}.

\begin{figure}[h]
\centering
\includegraphics[width=0.3\textwidth]{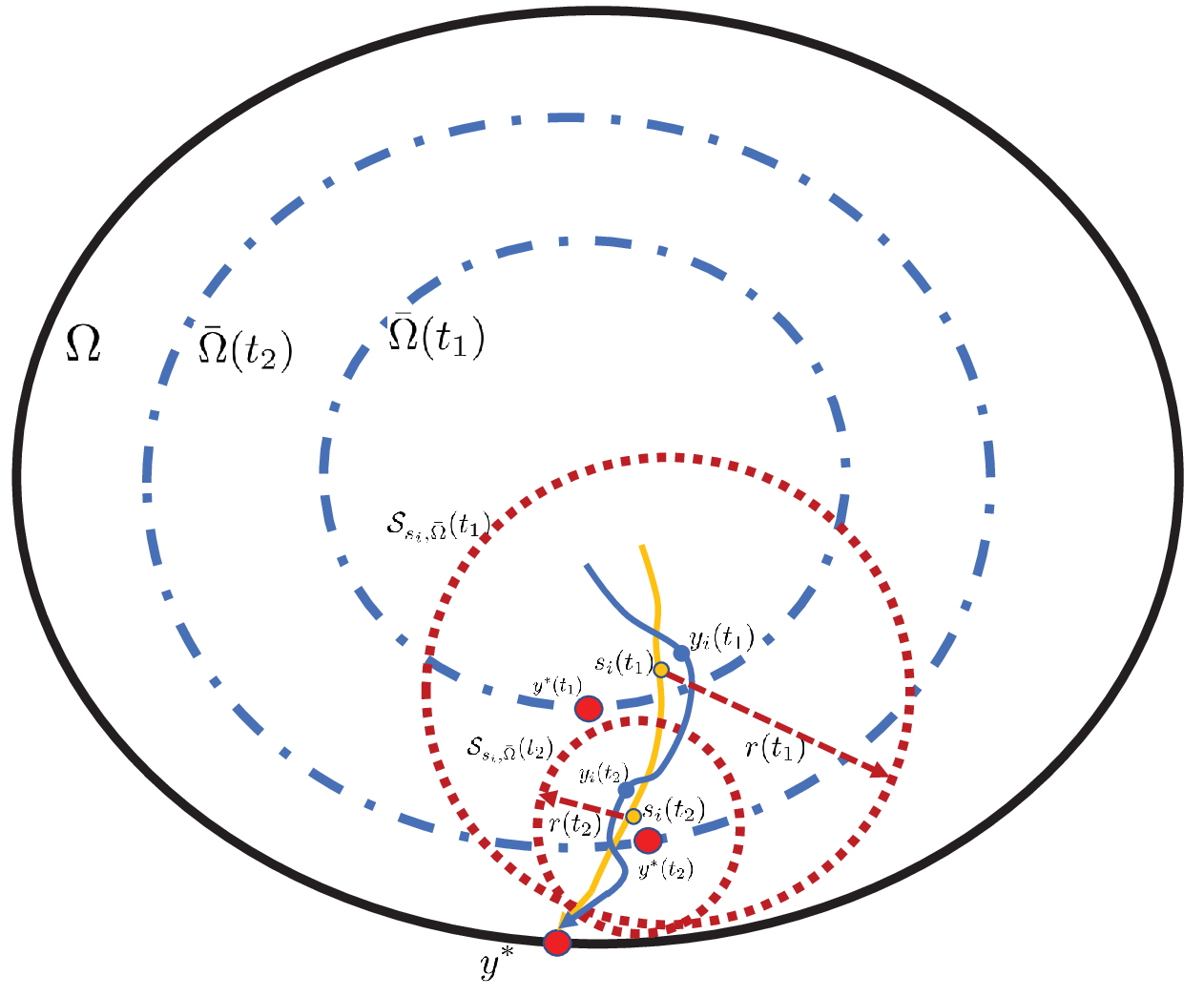}
\caption{Adaptive expanding set projection method. The blue line depicts the trajectory of the reference signal $s_i(t)$, the orange line shows the output trajectory $y_i(t)$, the blue dashed  circles represent the constraint sets $\bar\Omega(t_1)$ and $\bar\Omega(t_2)$ with $t_2>t_1\geq0$, and the red dashed  circles represent the safety sets $\mathcal{S}_{s_i,\bar\Omega}(t_1)$ and $\mathcal{S}_{s_i,\bar\Omega}(t_2)$.}\label{figg1a}
\end{figure}

\section{Convergence Analysis via Small-gain Theorem}\label{Convergenceanalysis}
In this section, we analyze the convergence for the closed-loop system composed of the MASs \eqref{sys} and the  protocol \eqref{24e1}
by small-gain theorem.

First, we need the following lemma to establish the equivalence between the equilibrium point of~\eqref{24e1} and the optimal solution $y^*$ to problem~\eqref{problem1}.
\begin{lemma}\label{lem4e1}
Let $s(\infty)$ denote the equilibrium point of \eqref{24e1}, $y^*$ denote the optimal solution to problem~\eqref{problem1}.
When the relations $\lim_{t\rightarrow\infty}\eta_{i}(t)=\lim_{t\rightarrow\infty}\eta_{j}(t)$, $\lim_{t\rightarrow\infty}z_{i}(t)=\lim_{t\rightarrow\infty}z_{j}(t)$, and $\lim_{t\rightarrow\infty}y_{i}(t)=s(\infty)$ hold, we have that the equilibrium point $s(\infty)$ of \eqref{24e1} equals to the optimal solution $y^*$ to problem~\eqref{problem1}.
\end{lemma}
\begin{proof}
 See Appendix \ref{2ap3}.
\end{proof}

For convenience, denote $D_{max}:=\max_{i\in\mathcal{V}}\|P_i\|\|\Pi_i\|^2$, $C_{max}:=\max_{i\in\mathcal{V}}\|C_i\|$ and $\lambda_{min}(P_i)$ denotes the smallest eigenvalue of the matrix $P_i$, and the positive definite matrix  $P_i$ is  a solution of the following Lyapunov equation:
\begin{eqnarray}\label{ria12}
(A_i+B_iK_{1i})^TP_i+P_i(A_i+B_iK_{1i})+cP_i=0
\end{eqnarray}
where
$c<-2\max_{i\in\mathcal{V}}\max_{1\leq j\leq n_i}\Re\{\lambda_j(A_i+B_iK_{1i})\}
$,~
~$
\gamma_{\theta}'=\max\Bigg\{4\gamma^{\dot{e}}_{\eta}(\lambda_N k_1+2+L_{max}\alpha)
\cdot\max_{i\in\mathcal{V}}\|C_i\Pi_i\| ,~ 2\gamma^{\dot{e}}_{\eta}\alpha\cdot\max_{i\in\mathcal{V}}\|C_i\Pi_i\|\Bigg\}$,~
$\gamma_{\eta}^{\eta^*} = N\beta_4(\alpha)/2 $, $
\gamma_{\eta} = \min\left\{\frac{\alpha^2}{k_1\lambda_2}, \frac{1}{2}(k_1\lambda_2 - 3), \frac{\beta_1(\alpha)}{2}\right\},$ $\gamma_{\eta}^{e} = \frac{N\beta_3(\alpha)}{2} + \frac{36\alpha^2L^4_{\max}}{k_2} $
$\gamma_{\eta}^{\dot{e}}= \frac{4L^2_{\max}}{k_2}$,
$\gamma_{x} =\frac{18D_{max}\beta^2\alpha^2L^2_{max}C^4_{max}}{c\lambda_{min}^2(P_i)},$
 $\gamma^{\widehat\eta}_{x}=\frac{2D_{max}}{c}(3\lambda_N^2 k_1^2+12+9\alpha^2L^2_{max}),$
$\gamma^{\widetilde z}_{x}=\frac{18D_{max}\alpha^2}{c}$, $
 \gamma_{\theta}^{\widetilde x}=\gamma^{e}_{\eta}\cdot\max_{i\in\mathcal{V}}\|C_i\|+\gamma^{\dot{e}}_{\eta}\cdot\max_{i\in\mathcal{V}}\|C_i(A_i+B_iK_{1i})\|\nonumber\\
+\gamma^{\dot{e}}_{\eta}\cdot\max_{i\in\mathcal{V}}\|C_i\Pi_i\|\cdot\max_{i\in\mathcal{V}}\|C_i\|\alpha L_{max}$.

It is time to present the main result of this section.

\begin{theorem}\label{the31}
Under Assumptions \ref{as1}-\ref{as6}, and given the MASs~\eqref{sys} with the distributed dynamic decision-making and control protocol \eqref{24e1},  if $\beta$ satisfies
\begin{eqnarray}\label{smaca}
   0<\beta<\min\left\{1,\frac{\gamma_{\eta}}{2\gamma_{\theta}'}, \sqrt{\frac{c}{4\gamma_x}},\frac{1}{4}\cdot\sqrt{\frac{c\gamma_{\eta}}{\gamma_{\theta}^{\widetilde x}\max(\gamma^{\widetilde z}_{x},2\gamma^{\widehat\eta}_{x})}}\right\}
\end{eqnarray}
then the \emph{optimal objective} is achieved,  i.e.,   $\lim_{t\rightarrow\infty} y_{i}(t)=y^*.$



\end{theorem}
\begin{proof}
Define the following variables:
\begin{equation}
\begin{aligned}
\theta_{i}(t)=\left[\begin{matrix}
\tilde{z}_{i}(t)\\
\tilde{\eta}_{i}(t)\\
\eta'(t)
\end{matrix}\right]=
\left[\begin{matrix}
z_i(t)-\bar{z}(t)\\
\eta_i(t)-\bar\eta(t)\\
\bar\eta(t)-\eta^*(t)
\end{matrix}\right],\quad \widetilde{x}_i(t)={x}_{i}(t)-\Pi_is_i(t)
\end{aligned}
\end{equation}
where $\bar{z}(t)=\frac{1}{N}\sum_{i=1}^{m}z_{i}(t)$, $\bar{\eta}(t)=\frac{1}{N}\sum_{i=1}^N\eta_i(t)$ and $\Pi_i$ is defined the same as in Lemma~\ref{lem31}. The variable
$\eta^*(t)$ is defined as
a solution to the following equations at time $t>0$:
\begin{equation}
\left\{\begin{aligned}\label{ads}
&\eta^*(t)=y^*(t)-\alpha \nabla f(y^*(t))\\
&P_{\bar{\Omega}(t)}(\eta^*(t))=y^*(t)
\end{aligned}
\right.
\end{equation}
where $y^*(t)$ denotes the time-varying optimal solutions to problem~\eqref{problem3}. Moreover, we stack $\theta_{i}$ and $\widetilde{x}_{i}$ for all $i\in \mathcal{V}$ as $\boldsymbol{\theta}(t)=\col({\theta}_1(t),\cdots,{\theta}_N(t))$ and $\boldsymbol{\widetilde{x}}(t)=\col({x}_1(t),\cdots,{x}_N(t))$, respectively.

Next, we establish the ISS for dynamic system described by \eqref{24e11}-\eqref{24e133}, as well as for the MAS \eqref{sys} with the input \eqref{24e12} in the following two inequities  whose proof are given as in Appendices \ref{2ap1} and \ref{2ap2}, respectively.

\begin{eqnarray}\label{sma3}
  \|\boldsymbol\theta(t)\|\!\!\!\!&\leq&\!\!\!\!-\beta_1(\|\boldsymbol\theta(0)\|,t)+2\sqrt{\frac{\gamma_{\theta}^{\widetilde x}}{\gamma_{\eta}}}\|\boldsymbol{\widetilde{x}}(t)\|\nonumber\\
\!\!\!\!&&\!\!\!\!+2\sqrt{\frac{\gamma^{\eta^*}_{\eta}}{\gamma_{\eta}}}(L_{max}+1)\|\dot{y}^*(t)\|
\end{eqnarray}
and
\begin{equation}
\|\boldsymbol{\widetilde{x}}(t)\|\leq-\beta_2(\|\boldsymbol{\widetilde{x}}(0)\|,t)+2\beta\sqrt{\frac{\max(\gamma^{\tilde z}_{x},2\gamma^{\hat\eta}_{x})}{c}}\|\boldsymbol{\theta}(t)\|\label{sma2}
\end{equation}
where the functions $\beta_1(\cdot,\cdot),~\beta_2(\cdot,\cdot):[0,+\infty)\times[0,+\infty)\rightarrow[0,+\infty)$ belong $\mathcal{KL}$-class functions.

According to the small-gain theorem \cite{jiang1994small}, since $ 2\sqrt{\frac{\gamma_{\theta}^{\widetilde x}}{\gamma_{\eta}}}\cdot2\beta\sqrt{\frac{\max(\gamma^{\widetilde z}_{x},2\gamma^{\widehat\eta}_{x})}{c}}<1$, the closed-loop system composed of the MASs \eqref{sys}, and the distributed dynamic decision-making and
control protocol \eqref{24e1}, is practically stable in terms of input-to-output, with $(\boldsymbol\theta(t),\boldsymbol{\widetilde{x}}(t))$ as the output and $\dot{y}^*(t)$ as the input.

By C2 in the design of the time-varying set $\bar{\Omega}(t),$ we have $\lim_{t\rightarrow\infty}\dot{y}^*(t)=0$. This implies   $\lim_{t\rightarrow\infty}\boldsymbol\theta(t)=0$ and $\lim_{t\rightarrow\infty}\boldsymbol{\widetilde{x}}(t)=0.$
Consequently, we have $\lim_{t\rightarrow\infty}s_i(t)-y^*=0$ and $\lim_{t\rightarrow\infty}y_i(t)-s_i(t)=0$. Thus, $\lim_{t\rightarrow\infty}y_i(t)-y^*=0$.
\end{proof}



\begin{corollary}
Under Assumptions \ref{as1}-\ref{as3}, consider the following distributed optimization algorithm:
\begin{align}
\dot {\eta}_i &=\beta k_1\sum_{j=1}^Na_{ij}(\eta_j-\eta_i)+\beta s_i-\beta\eta_i-\beta\alpha z_i\nonumber\\
\dot z_i &= \frac{ k_2}{\beta}\sum_{j=1}^Na_{ij}(z_j-z_i)+\frac{d\nabla f_{i}(s_{i})}{dt}\nonumber\\
s_i&=P_{\Omega}(\eta_i)\label{2ss4e1}
	\end{align}
where $z_i(0) = \nabla f_i(s_i(0))$.
By selecting suitable parameters for $k_1$, $k_2$, $\alpha$ and $\beta$, the trajectory of $s_i(t)$ will converge exponentially to the point $y^*$, which is a solution to the optimization Problem \eqref{problem1}.
\end{corollary}
\begin{remark}
It is noteworthy that, compared to the existing works \cite{zeng2016distributed,liang2017distributed}, the distributed optimization algorithm \eqref{2ss4e1} achieves exponential convergence. It also exhibits robustness to external disturbances, thereby offering a new approach to solving distributed constrained optimization problems that involve external disturbances, communication delays, and so forth.
\end{remark}


\section{Safety Analysis with Time-varying CBF}\label{sany}
In this section, we provide the safety analysis for the closed-loop system composed of the MASs \eqref{sys} and the  protocol \eqref{24e1}.

At first,  the following lemma is required.
\begin{lemma}\label{lem523}
Consider a safety zone for agent $i$'s state of MASs~\eqref{sys}:
\begin{equation}
\mathcal{R}_{s_i}(t)=\Big\{x_i\in \mathbb{R}^{n_i} \big| \|x_i(t)-\Pi_is_i(t)\|_{P_i}\leq \frac{\sqrt{\lambda_{min}(P_i)}r(t)}{\|C_i\|}\Big\}.
\end{equation}
If $x_i(t)\in\mathcal{R}_{s_i}(t)$, then we always have $y_i(t)\in\mathcal{S}_{s_i,\bar{\Omega}}(t)$. Moreover,  if $y_i(0)\in \int(\Omega)$, then $x_i(0)\in\mathcal{R}_{s_i}(0)$.
\end{lemma}
\begin{proof}
 See Appendix \ref{2ap4}.
\end{proof}

Given that the safety zone $\mathcal{R}_{s_i}(t)$ is a shrinking set, its rate of shrinkage plays a critical role in our safety analysis. Furthermore, as indicated by \eqref{uhas}, the shrinkage rate of $\mathcal{R}_{s_i}(t)$ is determined by the expansion rate of $\bar\Omega(t)$. Therefore, to ensure safety, it is imperative to design the \emph{expanding constraint set} $\bar\Omega(t)$ effectively. This design must meet the specified condition in addition to conditions C1-C3.
\begin{itemize}
\item[\textbf{C4:}] Denote  $\beta_1(t):=\max_{y\in\Omega}\min_{x\in\bar \Omega(t)} d(x,y)$ and
$\beta_2(t):=\min_{y\in\Omega}\min_{x\in \bar\Omega(t)} d(x,y)$. There are a $\xi$ not smaller than $1$, and two positive constants $\beta_1$ and $\beta_2$ such that the following two inequalities hold.

\begin{eqnarray}\label{disf}
\beta_1(t)\leq\xi \beta_2(t).
\end{eqnarray}
and
\begin{equation}
\left\{\begin{aligned}\label{aasfdds}
&\beta_1(t)\geq \beta_1e^{-vt},~\lim_{t\rightarrow\infty}\beta_1(t)=0\\
&\|\dot{y}^*(t)\|\leq \beta_2 e^{-vt}
\end{aligned}
\right.
\end{equation}
where $v$ is a positive constant and will be given in \eqref{defv}.

\end{itemize}

\begin{remark}
It is worth pointing out that \eqref{disf} holds when the \emph{expanding constraint set} $\bar\Omega(t)$ is uniformly expanded in any direction. For example, when $\Omega$ is a sphere,  $\bar\Omega(t)$ can be chosen as an expanding sphere that always has the
same ball center with $\Omega$. In this case, $\xi$ can be chosen as any positive constant not smaller than $1$. On the other hand, when the expansion speed of $\bar\Omega(t)$ is sufficiently slow, \eqref{aasfdds} holds for any positive constant $v$. Based on the above analysis,  C4 can be easily met in practice by choosing a suitable \emph{expanding constraint set} $\bar\Omega(t)$.

\end{remark}

With the above condition, we  construct a time-varying CBF function in the following lemma.

\begin{lemma}\label{smaas}
   Construct  the following function:
\begin{flalign}\label{cbf9}
h(\boldsymbol x,\boldsymbol s,t):=\sum_{i=1}^N h_i(x_i,s_i,t)
\end{flalign}
where $h_i(x_i,s_i,t):=\beta_2(t)-\widetilde{x}^T_{i}(t)P_i\widetilde{x}_{i}(t).$
The function $h(\boldsymbol x,\boldsymbol s,t)$ is a  time-varying CBF,
if $\beta$ satisfies \eqref{smaca} and the following inequality simultaneously.
\begin{flalign}
\beta<\min\Bigg\{&\left(\frac{c^2}{72\beta^2\alpha^2L^2_{max}\kappa_{max}(P)\Pi_{max}^2C^2_{max}}\right)^{1/4}\nonumber \\
&~\sqrt{\frac{cH_1}{8\beta_1\xi M'}} ,~\frac{c}{8\gamma_{\theta}^{\tilde x}},\frac{\gamma_{\theta}}{8\max\{\gamma^{\widehat\eta}_{x},2\gamma^{\boldsymbol{\widetilde z}}_{x}\}},\frac{c}{2\gamma_{\eta}}\Bigg \}
\end{flalign}
and
\begin{equation}\label{defv}
    v=\frac{\beta\gamma_{\eta}}{4}
\end{equation}
where   $H_1=\|\boldsymbol\theta(0)\|^2/2+\frac{4\gamma^{\eta^*}_{\eta}(L_{max}+1)h^2}{\gamma_{\eta}}$ and $M'=\frac{4}{c}P_{max}\Pi_{max}^2\left[(3\lambda_N^2 K_1^2+12+9\alpha^2L^2_{max})+9\alpha^2\right].$
\end{lemma}

\begin{proof}
See Appendix \ref{2ap5}.
\end{proof}

\begin{theorem}\label{lem46}
Suppose that Assumptions \ref{as1}-\ref{as42} hold and the initial output of each agent $i$ satisfies $y_i(0)\in \int(\Omega)$.  Considering the MASs~\eqref{sys} with the distributed dynamic decision-making
and control protocol \eqref{24e1}, the \emph{safety objective} is achieved, i.e.,  $y_i(t)\in\Omega$ for all $t\geq 0$.

\end{theorem}

\begin{proof}
By applying time-varying CBF techniques and using Lemmas \ref{leas} and \ref{smaas}, we have
$x_i(t)\in\mathcal{R}_{s_i}(t)$.
By Lemma \ref{lem523}, $y_i(t)\in \mathcal{S}_{s_i,\bar\Omega}(t)\subseteq\Omega.$ This guarantees safety.
\end{proof}
\begin{remark}
Here we propose a novel approach to address the safety control problem, particularly when the expected convergence point is positioned on the safety set's boundary. Specifically,  our approach constructs a shrinking set within the safety set to approximate the convergence point accurately. A safety reference controller is then developed to confine the system state within this shrinking set. This approach differs from existing QP methods \cite{xu2015robustness} by preserving the safety objective without compromising control objectives. Detailed discussion of this aspect, however, is beyond the scope of this paper.

\end{remark}

\section{Simulation}\label{exam}
Consider a  continuous-time  MAS with five agents whose dynamics are described as follows.
\begin{equation*}
    A_i=\left[
	\begin{matrix}
		0 & 1 & 0 \\
		0 & 0 & 1 \\
		-1 & -2 & -2-i
	\end{matrix}
	\right]
	, B_i=\left[
	\begin{matrix}
		1 & 0 \\
		0 & 1 \\
		1 & 1
	\end{matrix}
	\right]
	, C_i=\left[
	\begin{matrix}
		1 & 0 & 0\\
		0 & 1 & 0\\
	\end{matrix}
	\right]
\end{equation*}
where $i=1,\ldots,5$. The connectivity weights in the graph $\mathcal{G}$ are  $a_{12}=a_{21}=a_{23}=a_{32}=a_{34}=a_{43}=a_{45}=a_{15}=a_{51}=1;$ $a_{ij}=0$, otherwise.
The objective function is $f(r)=\sum_{i=1}^{5} f_i(r)=\frac{1}{2}\|r-e_i\|^2,$ where $e_i=\col(i,i),~i=1,2,3,4$. The constraint set $\Omega:=\left\{x \in \mathbb{R}^2 | \|x\| \leq 2\right\}.$
The initial outputs are $y_1(0)=\col(0,0),y_2(0)=\col(1.1,1.1),y_3(0)=\col(0.1,1),y_4(0)=\col(-0.1,0.6),y_5(0)=\col(-1,-1.1).$

Choose the gain matrices $K_{1i}$ and $K_{2_i}$ as follows.
\begin{align*}
	 K_{11}=\left[\begin{matrix}
			-1.63 & 0.97 &  0.24\\
			0.28&  -3 & -1.54
		\end{matrix}\right],	&~ K_{21}=\left[\begin{matrix}
			1.7& -1.79\\
			-0.42&2.65
		\end{matrix}\right]\\
  K_{12}=\left[\begin{matrix}
			-1.56 &1.04& 0.56\\
			0.53 &-2.6& -1.37
		\end{matrix}\right], &~K_{22}=\left[\begin{matrix}
			1.68 &  -1.7\\
			-0.61 &   2.41
		\end{matrix}\right]\\
K_{13}=\left[\begin{matrix}
			-1.48&1.15&0.92\\
			0.719&-2.34&-1.11
		\end{matrix}\right],
  	&~ K_{23}=\left[\begin{matrix}
			1.63  & -1.69\\
			-0.72  &  2.28
		\end{matrix}\right]\\
	 K_{14}=\left[\begin{matrix}
			-1.33&1.33&1.38\\
			0.78&-2.2&-0.84
		\end{matrix}\right],	&~ K_{24}=\left[\begin{matrix}
			1.52&-1.73\\
			-0.76&2.26
		\end{matrix}\right]\\
  	K_{15}=\left[\begin{matrix}
			-3.09&1.48&0.86\\
			2.41&-1.12&1.36
		\end{matrix}\right],	& ~K_{25}=\left[\begin{matrix}
			3.2&-2.16\\
			-2.11&2.01
		\end{matrix}\right].\\
	\end{align*}
It can be verified that  $A_i+B_i K_i,~i=1,\cdots,5,$ are Hurwize, and $K_{2i}=\Psi_i-K_{1i}\Pi_i$.  The \emph{ expanding constraint set} $\bar\Omega(t)=\left\{x \in \mathbb{R}^2 \|x\| \leq 2-0.9e^{-0.5t}\right\}$ and the \emph{shrinking safety set} $\mathcal{S}_{s_i,\bar\Omega}(t)$ is $\mathcal{B}(s_i(t),r(t))$, where $r(t)=0.9e^{-0.5t}$.
	The parameters are chosen as $\alpha=0.1,~K_1=5,k_2=10$ and $\beta=0.1$. The trajectories of agents' output $y_i$ and the reference signal $s_i$ are shown in Fig. \ref{fig2}. It is shown that all the outputs $y_i$ and converge to the optimal solution $y^*=(\sqrt{2},\sqrt{2})$ and remain within $\Omega.$
	\begin{figure}[h]
		\centering
		\includegraphics[width=0.5\textwidth]{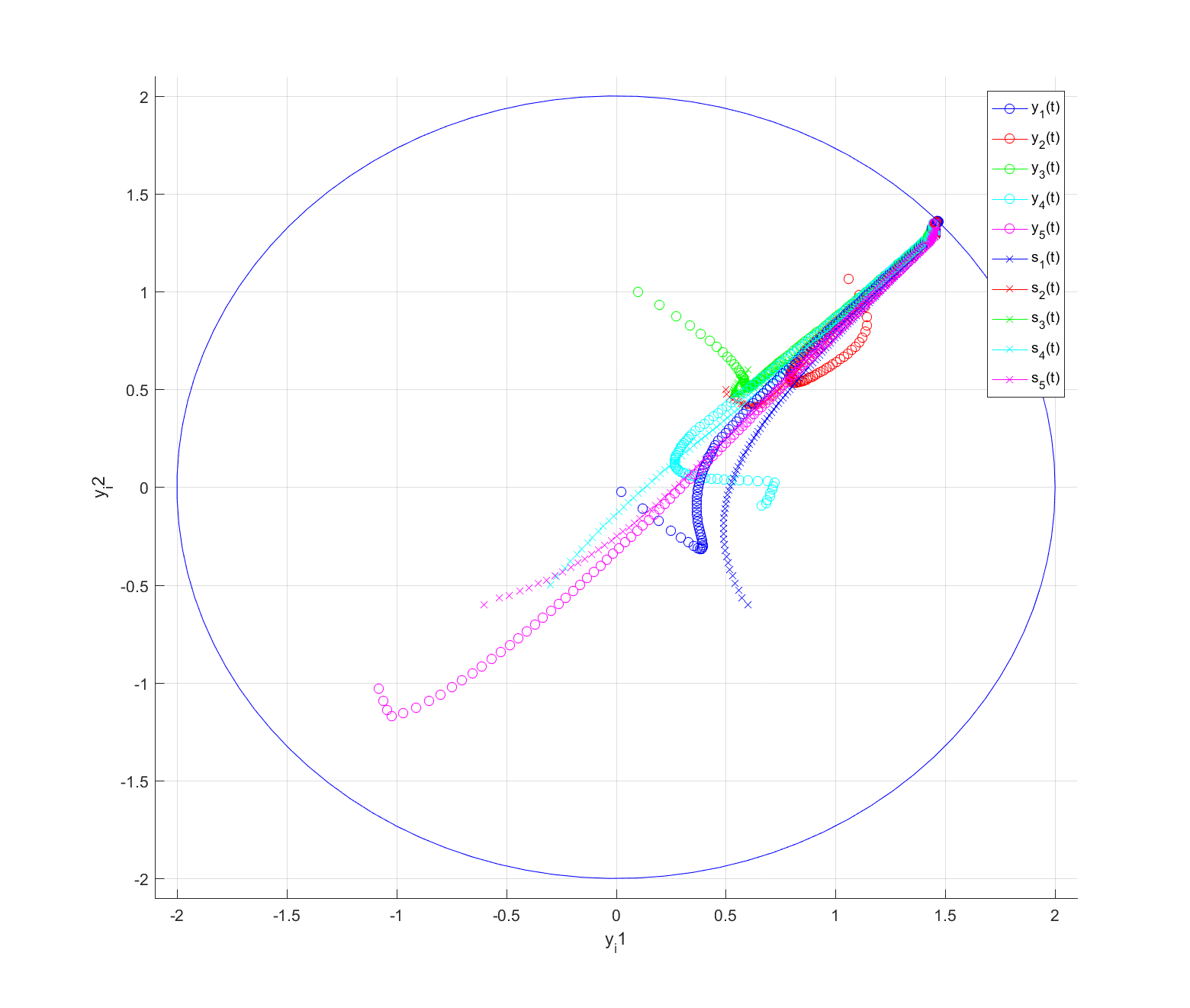}
		\caption{The trajectories of agents' output $y_i$ and the reference signal $s_i$}\label{fig2}
	\end{figure}

\section{Conclusion}\label{con}
In this paper, we investigated the distributed
optimal output consensus control problem of linear heterogeneous MASs subject to safety constraints.
By proposing an adaptive expanding set projection
method, we developed a novel dynamic decision-making and control protocol that achieve both convergence and safety requirements simultaneously, even when the optimal
solution lies at the boundary of the safety zone. We also established the convergence analysis and safety analysis of the closed-loop system by small-gain theorem and time-varying CBF theory.
Future works may consider nonlinear MASs and the case with different safety constraints. Another direction is the distributed safety optimization with dynamic objective functions.

\appendices
\setcounter{equation}{0}
\renewcommand\theequation{\Roman{section}.\arabic{equation}}
\renewcommand\thelemma{\Roman{section}.\arabic{lemma}}
\renewcommand\theclaim{\Roman{section}.\arabic{claim}}

 \section{{Proof~of~Lemma~\ref{lem4e1}}}\label{2ap3}

Since  $z_i(0) = \nabla f_i(y_i(0))$, we have $\bar{z}(0) = \frac{1}{N}\sum_{i=1}^N\nabla f_i(y_i(0))$. Then, with~\eqref{24e11} for all $t>0$, we obtain $\dot{\bar{z}}(t)=\frac{1}{N}\sum_{i=1}^N\frac{d}{dt}\nabla f_i(y_i(t)),$ which further implies
$\bar{z}(t)=\frac{1}{N}\sum_{i=1}^N\nabla f_i(y_i(t)),~\forall t\geq0.$

By defining $\eta(\infty)=\lim_{t\rightarrow\infty}\eta_{i}(t)$ and  $z(\infty)=\lim_{t\rightarrow\infty}z_{i}(t)$ for all $i\in{\mathcal{V}}$, $s(\infty)=\lim_{t\rightarrow\infty}y_{i}(t)$  implies
\begin{equation}
z(\infty)=\frac{1}{N}\sum_{i=1}^N\nabla f_i\left(s(\infty)\right)=\nabla f\left(s(\infty)\right).\label{zinfinity}
\end{equation}

Since  $\lim_{t\rightarrow\infty}\dot {\eta}_i(t)=0$ and $\lim_{t\rightarrow\infty}\eta_i(t)=\lim_{t\rightarrow\infty}{\eta}_j(t)$, we have the following equality from~\eqref{24e122} and~\eqref{zinfinity}:
\begin{equation}\label{adss}
s(\infty)-\alpha \nabla f\left(s(\infty)\right)=\eta(\infty).
\end{equation}
Taking the projection on both sides of~\eqref{adss} and using the fact $s(\infty)=P_{\Omega}\left(\eta(\infty)\right)$, we arrive at
\begin{equation}\label{adsaa}
P_{\Omega}\left(s(\infty)-\alpha \nabla f\left(s(\infty)\right)\right)=s(\infty).
\end{equation}
Then the equilibrium point $s(\infty)$ of \eqref{24e1} equals to the optimal solution $y^*$ to problem~\eqref{problem1}.

\section{{Proof~of~Inequality~\eqref{sma3}}}\label{2ap1}
The following notations are used throughout this proof.

Denote $\widetilde{z} _i(t):=z_i(t)-\bar{z}(t)$, $\bar{z}(t):=\sum_{i=1}^Nz_i(t)/N,$
$\bar\eta(t):=\sum_{i=1}^N\eta_i(t)/N$, $\widetilde\eta_i(t):=\eta_i(t)-\bar\eta(t)$,  $e_i(t):=y_i(t)-P_{\bar\Omega(t)}(\eta_i(t))$, $\eta'(t):=\bar\eta(t)-\eta^*(t)$, $y^*(t):=P_{\bar\Omega(t)}(\eta^*(t))$, $b(t) := y^*(t) - \eta^*(t)$, $z^*(t):=\sum_{i=1}^N\nabla f_i(y^*(t))/N$, $\widehat z_i(t):=z_i(t)-z^*(t)$,  $\widehat{\eta}_i(t):=\eta_i(t)-\eta^*(t)$, $\bar{s}(t)=\frac{1}{N}\sum_{j=1}^Ns_i(t)$, $s'(t)=P_{\bar\Omega(t)}(\bar{s}(t))$ $\widetilde{s}_i(t)=s_i(t)-\bar{s}(t)$ and $\widehat{s}_i(t):=s_i(t)-y^*(t)$. Let $\boldsymbol{\widetilde\eta}(t):=\col(\widetilde\eta_1(t),\cdots,\widetilde\eta_N(t))$, $\boldsymbol{\eta}(t):=\col(\eta_1(t),\cdots,\eta_N(t))$,   $\boldsymbol{\widetilde{z}}(t):=\text{col}(\widetilde{z}_1(t),\cdots,\widetilde{z}_N(t))$, $\boldsymbol{y}(t):=\text{col}(y_1(t),\cdots,y_N(t))$, $\boldsymbol{{e}}(t):=\text{col}({e}_1(t),\cdots,{e}_N(t))$, $\boldsymbol{\widetilde{s}}(t):=\text{col}(\widetilde{s}_1(t),\cdots,\widetilde{s}_N(t)),$
$\boldsymbol{\widehat{z}}(t):=\text{col}(\widehat{z}_1(t),\cdots,\widehat{z}_N(t))$
and $\boldsymbol{\widehat{\eta}}(t):=\text{col}(\widehat{\eta}_1(t),\cdots,\widehat{\eta}_N(t))$.

First,  we will examine the consensusability of the variables $\eta_i(t)$ and $z_i(t)$ in the following two lemmas.
\begin{lemma}\label{lem41}
 Consider  the  algorithm  \eqref{24e1} and construct the Lyapunov  function as follows.
\begin{eqnarray}
    V_{\boldsymbol{\widetilde{z}}}(t)=\sum_{i=1}^N\|\widetilde{z}_i(t)\|^2/2.
\end{eqnarray}
 Then
\begin{flalign}\label{asfaaasd1}
   \frac{dV_{\boldsymbol{\widetilde{z}}}(t)}{dt}\leq&-\left(k_2\lambda_2-\frac{72\alpha^2L^2_{max}}{k_2}\right)\beta V_{\boldsymbol{\widetilde{z}}}(t)
   \nonumber\\
&+\frac{4L^2_{max}}{k_2}\left(6\lambda_N^2 k_1^2+24+18\alpha^2L^2_{max}\right)\beta\nonumber\\
 &\times(\|\boldsymbol{\widetilde\eta}(t)\|^2+N\|\eta'(t)\|^2)+\frac{36\alpha^2L^4_{max}\beta}{k_2}\|\boldsymbol{e}(t)\|^2\nonumber\\
 &+\frac{4L^2_{max}\beta}{k_2}\|\dot {\boldsymbol{e}}(t)\|^2
\end{flalign}
where   $\lambda_2$ is the smallest eigenvalue of the Laplacian matrix $\mathcal{L}_{\mathcal{G}}$ for the graph $\mathcal{G},$ and $L_{max}:=\max_{i\in\mathcal{V}}L_i.$
\end{lemma}
\begin{proof}
Denote  $$p_i:=\frac{\partial ^{2}f_i}{\partial x^{2}}\Big|_{x=y_i(t)} \dot y_i(t).$$ Then, it follows from the second equation in \eqref{24e1} that the derivative of $\bar{z} (t)$ satisfies that
\begin{align}\label{khkkaaaas}
	\dot{\bar z}(t)=\sum_{i=1}^{N}p_i(t)/N.
\end{align}

Since $\widetilde{z} _i(t)=z_i(t)-\bar{z}(t)$, it follows from \eqref{24e1} that
\begin{flalign}\label{ojs}
	\dot{\widetilde z}_i(t)&=\dot{z}_i(t)-\dot{\bar{z}}(t)\nonumber\\
	&=\frac{k_2}{\beta} \sum_{j=1}^Na_{ij}(\widetilde{z}_j(t)-\widetilde{z}_i(t))+p_i(t) - \sum_{i=1}^{N} p_i(t)/N\nonumber\\
 &=\frac{k_2}{\beta}\sum_{j=1}^Na_{ij}(\widetilde{z}_j(t)-\widetilde{z}_i(t))+\widetilde{p}_i(t)
\end{flalign}
where $\widetilde{p}_i(t):=p_i(t)-\sum_{j=1}^Np_j(t)/N,~1\leq i\leq N.$

By \eqref{ojs}, we have
\begin{eqnarray}\label{uwdf}
    \frac{dV_{\boldsymbol{\widetilde{z}}}(t)}{dt}\!\!\!\!&\leq&\!\!\!\! \sum_{i=1}^N\left\langle  \widetilde{z} _i(t),    \frac{k_2}{\beta} \sum_{j=1}^Na_{ij}(\widetilde{z}_j(t)-\widetilde{z}_i(t))  \right \rangle                            \nonumber\\
    \!\!\!\!&&\!\!\!\! +\sum_{i=1}^N\left\langle  \widetilde{z} _i(t),   \widetilde{p}_i(t)  \right \rangle                            \nonumber\\
        \!\!\!\!&\leq&\!\!\!\!-\frac{2k_2\lambda_2}{\beta} V_{\boldsymbol{\widetilde{z}}}(t)+\sum_{i=1}^N\left\langle  \widetilde{z} _i(t),   \widetilde{p}_i(t)  \right \rangle  \nonumber\\
    \!\!\!\!&\leq&\!\!\!\!-\frac{k_2\lambda_2}{\beta} V_{\boldsymbol{\widetilde{z}}}(t)+\frac{ \beta}{2  k_2}\|\boldsymbol{\widetilde{p}}(t)\|^2
\end{eqnarray}
where $\boldsymbol{\widetilde{p}}(t):=\text{col}(\widetilde{p}_1(t),\cdots,\widetilde{p}_N(t))$.

 Since $y^*(t)-\eta^*(t)-z^*(t)=0,$   it follows from  \eqref{24e1} that
\begin{flalign}\label{ahuhfa}
\dot {\eta}_i(t)=&\beta k_1\sum_{j=1}^Na_{ij}(\widehat\eta_j(t)-\widehat\eta_i(t))\nonumber\\
&+\beta \left(\widehat s_i(t)-\widehat \eta_i(t)-\alpha \widehat z_i(t)\right).
\end{flalign}

Clearly,
\begin{flalign}\label{kajdaw}
\|P_{\bar\Omega(t)}(\eta_i(t))-P_{\bar\Omega(t)}(\eta^*(t)))\leq&\|\eta_i(t)-\eta^*(t)\|=\|\widehat \eta_i(t)\|.
\end{flalign}

For any $a_1,~a_2,~\cdots,a_k\in\mathbb{R}^q$, where $k\in\mathbb{N}^{+},$
$\left\|\sum_{i=1}^ka_i\right\|^2\leq k\sum_{i=1}^k\|a_i\|^2.$
Based on this inequality and \eqref{kajdaw}, we have
\begin{equation}
\begin{aligned}
&\|\widehat s_i(t)-\widehat \eta_i(t)\|^2\leq 2\|\widehat s_i(t)\|^2+2\|\widehat \eta_i(t)\|^2\\
&~~~~~~~~\leq 2\|P_{\bar\Omega(t)}(\eta_i(t))-P_{\bar\Omega(t)}(\eta^*(t)))\|^2+2\|\widehat \eta_i(t)\|^2 \\
&~~~~~~~~\leq 4\|\widehat \eta_i(t)\|^2.\label{Theorem34}
\end{aligned}
\end{equation}

It follows from \eqref{ahuhfa} that
\begin{flalign}\label{uhww}
   \|\dot{\boldsymbol{\eta}}(t)\|^2\leq&3\lambda_N^2 k_1^2 \|\boldsymbol{\widehat\eta}(t)\|^2+3\sum_{i=1}^N\|\beta (\widehat s_i(t)-\widehat \eta_i(t))\|^2\nonumber\\
&+3\|\alpha\beta \boldsymbol{\widehat z}(t)\|^2\nonumber\\
  \leq&(3\lambda_N^2 k_1^2\!+\!12)\beta^2 \|\boldsymbol{\widehat\eta}(t)\|^2+3\alpha^2\beta^2\|\boldsymbol{\boldsymbol{\widehat z}}(t)\|^2.
\end{flalign}

Since $z^*(t)=\sum_{i=1}^N\nabla f_i(y^*(t))/N$ and $\bar z(t)=\sum_{i=1}^N\nabla f_i(y_i(t))/N$,
we have
\begin{flalign}
  \|\bar z(t)-z^*(t)\|^2\leq&\frac{2}{N^2}\Big\|\sum_{i=1}^N\nabla f_i(y_i(t))-\sum_{i=1}^N\nabla f_i(s_i(t))\Big\|^2\nonumber\\
& +\frac{2}{N^2}\Big\|\sum_{i=1}^N\nabla f_i(s_i(t))-\sum_{i=1}^N\nabla f_i(y^*(t))\Big\|^2\nonumber\\
\leq&\frac{2L^2_{max}}{N^2}\Big(N\|\boldsymbol{e}(t)\|^2+\sum_{i=1}^N\|s_i(t)-y^*(t)\|^2\Big)\nonumber\\
\leq&\frac{2L^2_{max}}{N}(\|\boldsymbol{e}(t)\|^2+\|\boldsymbol{\widehat\eta}(t)\|^2).
\end{flalign}
It implies
\begin{flalign}\label{jklaw}
  \|\widehat{z}_i(t)\|^2\leq&3\|\widetilde z_i(t)\|^2+\frac{3}{2}\|\bar z(t)-z^*(t)\|^2\nonumber\\
   \leq&3\|\widetilde z_i(t)\|^2+\frac{3L^2_{max}}{N}(\|\boldsymbol{e}(t)\|^2\!+\!\|\boldsymbol{\widehat\eta}(t)\|^2).
\end{flalign}

Since $\boldsymbol{y}(t)=\boldsymbol{\eta}(t) +\boldsymbol{e}(t)$, we have
\begin{flalign}
\label{kaiiw}
\|\dot {\boldsymbol{y}}(t)\|^2\leq& 2\|\dot {\boldsymbol{e}}(t)\|^2+2\|\dot{\boldsymbol{\eta}}(t)\|^2\nonumber\\
 \leq&\left(6\lambda_N^2 k_1^2+24+18\alpha^2L^2_{max}\right)\beta^2\|\boldsymbol{\widehat\eta}(t)\|^2\nonumber\\
&+18\alpha^2L^2_{max}\beta^2\|\boldsymbol{e}(t)\|^2\nonumber\\
&+18\alpha^2\beta^2\|\boldsymbol{\widetilde z}(t)\|^2+2\|\dot{\boldsymbol{e} }(t)\|^2.
\end{flalign}

Since $f_i$ is $L_i$ smooth, the Hessian Matrix $\frac{\partial ^{2}f_i}{\partial x^{2}}\leq L_i I_{p},~\forall x\in \mathbb{R}^p.$
It implies
\begin{flalign}\label{oajjw}
\|\boldsymbol{\widetilde{p}}(t)\|^2\leq 4\sum_{i=1}^N\|p_i(t)\|^2\leq 4L^2_{max}\|\dot{\boldsymbol{y}} (t)\|^2.
\end{flalign}

It follows from \eqref{uwdf}, \eqref{kaiiw} and \eqref{oajjw} that
\begin{flalign}\label{asfaaas}
   \frac{dV_{\boldsymbol{\widetilde{z}}}(t)}{dt}\leq&-\left(\frac{k_2}{\beta}\lambda_2-\frac{72\alpha^2\beta^3L^2_{max}}{k_2}\right) V_{\boldsymbol{\widetilde{z}}}(t)
   \nonumber\\
&+\frac{2L^2_{max}}{k_2}\left(6\lambda_N^2 k_1^2+24+18\alpha^2L^2_{max}\right)\beta^3\|\boldsymbol{\widehat\eta}(t)\|^2\nonumber\\
 &+\frac{36\alpha^2L^4_{max}\beta^3}{k_2}\|\boldsymbol{e}(t)\|^2+\frac{4L^2_{max}\beta}{k_2}\|\dot{\boldsymbol{e}}(t)\|^2.
 \end{flalign}

Since $0<\beta<1$, \eqref{asfaaas} implies \eqref{asfaaasd1}.
\end{proof}

\begin{lemma}\label{lem42}
    Consider the  algorithm  \eqref{24e1} and construct the following Lyapunov function.
\begin{equation}
    V_{\boldsymbol{\widetilde{\eta}}}(t)=\frac{1}{2}\sum_{i=1}^N\|\widetilde{\eta}_i(t)\|^2.
\end{equation}
Then we have
\begin{equation}\label{aklew}
\begin{aligned}
   \frac{d V_{\boldsymbol{\widetilde{\eta}}}(t)}{dt}\leq-\beta(k_1\lambda_2-3) V_{\boldsymbol{\widetilde{\eta}}}(t)+\frac{\beta\alpha^2}{2k_1\lambda_2}\|\boldsymbol{\widetilde{z}}(t)\|^2.
\end{aligned}
\end{equation}

\end{lemma}

\begin{proof}
It follows from \eqref{24e122} that
\begin{equation}
    \dot{\bar{\eta}}(t)=\bar{s}(t)-\bar{\eta}(t)-\alpha\bar{z}(t).
\end{equation}

Recalling that $\widetilde{\eta}_i=\eta_i(t)-\bar{\eta}(t)$ and $\widetilde{s}_i(t)=s_i(t)-\bar{s}(t)$, we obtain
\begin{flalign}\label{kau3}
\dot{\widetilde{\eta}}_i(t)&=\beta k_1\sum_{j=1}^Na_{ij}(\widetilde{\eta}_j(t)-\widetilde{\eta}_i(t)) \nonumber\\
&\quad+\beta\widetilde{s}_i(t)-\beta\widetilde{\eta}_i(t)-\beta\alpha\widetilde{z}_i(t).
\end{flalign}

Equation \eqref{kau3} can be compactly expressed as follows.
\begin{equation}\label{kau23}
\begin{aligned}
  \dot{\boldsymbol{\widetilde{\eta}}}(t)=k_1(\mathcal{L}_{\mathcal{G}}\otimes I_p)\boldsymbol{\widetilde{\eta}}(t)+\boldsymbol{\widetilde{s}}(t)-\boldsymbol{\widetilde{\eta}}(t)-\alpha\boldsymbol{\widetilde{z}}(t).
\end{aligned}
\end{equation}

This implies
\begin{equation}\label{radf}
\begin{aligned}
   \frac{d V_{\boldsymbol{\widetilde{\eta}}}(t)}{dt}\leq-2\beta k_1\lambda_2 V_{\boldsymbol{\widetilde{\eta}}}(t)+\beta\langle{\widetilde{\eta}}(t),~\boldsymbol{\widetilde{s}}(t)-\boldsymbol{\widetilde{\eta}}(t)-\alpha\boldsymbol{\widetilde{z}}(t)\rangle.
\end{aligned}
\end{equation}

Using the convexity and Young inequality, we have

\begin{flalign}\label{ukw}
    &~~~~\langle \boldsymbol{\widetilde{\eta}}(t),~\boldsymbol{\widetilde{s}}(t)-\boldsymbol{\widetilde{\eta}}(t)\rangle \nonumber\\
    &\leq -\frac{1}{2}\|\boldsymbol{\widetilde{\eta}}(t)\|^2+\frac{1}{2}\|\boldsymbol{\widetilde{s}}(t)\|^2 \nonumber\\
    &\leq-\frac{1}{2}\|\boldsymbol{\widetilde{\eta}}(t)\|^2+\frac{1}{2N}\sum_{i=1}^N\sum_{j=1}^N\|s_i(t)-s_j(t)\|^2 \nonumber\\
    &\leq-\frac{1}{2}\|\boldsymbol{\widetilde{\eta}}(t)\|^2+\frac{1}{2N}\sum_{i=1}^N\sum_{j=1}^N\|\eta_i(t)-\eta_j(t)\|^2 \nonumber\\
     & \leq\frac{3}{2}\|\boldsymbol{\widetilde{\eta}}(t)\|^2.
\end{flalign}

Applying the Young inequality again, we derive
\begin{equation}\label{afa}
    -\langle{\boldsymbol{\widetilde{\eta}}}(t),~\alpha\boldsymbol{\widetilde{z}}(t)\rangle\leq\frac{k_1\lambda_2}{2}\|\boldsymbol{\widetilde{\eta}}(t)\|^2+\frac{\alpha^2}{2k_1\lambda_2}\|\boldsymbol{\widetilde{z}}(t)\|^2.
\end{equation}

Substituting \eqref{ukw} and \eqref{afa} into \eqref{radf}, we conclude that \eqref{aklew} holds.
\end{proof}

Then we have the following lemma.

\begin{lemma}\label{the1}
Construct the Lyapunov function as follows.
\begin{align}\label{I22}
V_{\boldsymbol\theta}(t)= V_{\boldsymbol{\widetilde z}}(t) + V_{\boldsymbol{\widetilde \eta}}(t) +N V_{\eta'}(t)/2 \nonumber\\
\end{align}
where
\begin{flalign}
V_{\eta'}(t) &= V_1(t) + V_2(t)\nonumber \\
V_1(t) &= \frac{1}{2} \| s'(t) - \bar\eta(t) - b(t) \|^2\nonumber \\
V_2(t) &= \frac{1}{2} \| \bar\eta(t) - \eta^*(t) \|^2 ) \label{adsaa}
\end{flalign}
Given the distributed optimization algorithm \eqref{24e1}, we obtain
\begin{align}\label{jkasc}
\frac{d V_{\boldsymbol\theta}(t)}{dt} &\leq -\beta\gamma_{\eta}V_{\boldsymbol\theta}(t) + \beta\gamma^{\eta^*}_{\eta}\|\dot\eta^*(t)\|^2 \nonumber \\
&\quad + \beta\gamma^{e}_{\eta}\|\boldsymbol{e}(t)\|^2 + \gamma^{\dot{\boldsymbol{e}}}_{\eta}\beta\|\dot{e}(t)\|^2.
\end{align}
\end{lemma}

\begin{proof}
First, we will prove that the derivative of $V_{\eta'}(t)$ satisfies the following inequality by considering two cases.

\begin{align}\label{khakw}
\frac{d V_{\eta'}(t)}{dt} &\leq -\beta_1(\alpha) \beta V_{\eta'}(t) + \beta_2(\alpha)\beta\|\boldsymbol{\widetilde\eta}(t)\|^2 \nonumber \\
&\quad + \beta_3(\alpha)\beta\|\boldsymbol{e}(t)\|^2 + \beta_4(\alpha)\beta\|\dot{\eta}^*(t)\|^2.
\end{align}

$\bold{Case~1:~s'(t)\in \partial\bar\Omega(t).}$


According to the equation \eqref{24e1}, we have
 \begin{eqnarray}\label{jkas}
    \dot{ \bar{\eta}}(t)\!\!\!\!&=&\!\!\!\!\beta\bar{s}(t)-\beta\bar{\eta}(t)-\beta\alpha\bar z(t)\nonumber\\
    \!\!\!\!&=&\!\!\!\!\beta s'(t)-\beta \bar{\eta}(t)-\alpha\beta \nabla f(s'(t))\nonumber\\
   \!\!\!\!&&\!\!\!\! +\beta \Delta_s(t)+\beta \Delta_{z}(t).
 \end{eqnarray}

Next, it follows from \eqref{adsaa} that the derivative of $V_1(t)$ satisfies
\begin{eqnarray}\label{kladk}
    \frac{dV_1(t)}{dt}=\langle s'(t)-\bar\eta(t)-b(t),  \dot{ s}'(t)-\dot{\bar \eta}(t)-\dot{b}(t) \rangle.
\end{eqnarray}

Furthermore, let's define the tangent space and the cotangent space for $\partial\bar \Omega(t)$ at a point $s\in\partial \bar\Omega(t)$ as follows:
\begin{eqnarray}
 &&\mathcal{T}_{s}=\{x\in\mathbb{R}^q~\big|x=ky,~y=\lim_{n\rightarrow\infty}\frac{s_n-s}{\|s_n-s\|},~\{s_n\}  \nonumber\\
    &&~~~~~~~~~~~~~~~~~~~~~~~\text{is~a~sequence~belonging~to}~\partial\bar \Omega(t)\nonumber\\
&&~~~~~~~~~~~~~~~\text{and~satisfying~that}~\lim_{n\rightarrow\infty}s_n=s,~k\in\mathbb{R}\}\nonumber\\
    &&\mathcal{D}_{s}=\{x\in\mathbb{R}^q~\big|x^Ty=0,~\forall y \in\mathcal{T}_{s}\}.
\end{eqnarray}

According to the above definition, for any $s\in\partial \bar\Omega(t)$,
\begin{eqnarray}\label{kkhad}
P_{\mathcal{T}_{s}}(x+y)\!\!\!\!&=& \!\!\!\!P_{\mathcal{T}_{s}}(x),~\forall~x\in\mathbb{R}^q~,y\in\mathcal{D}_{s}\nonumber\\
   P_{\mathcal{T}_{s}}(x+y)\!\!\!\!&=& \!\!\!\!P_{\mathcal{T}_{s}}(x)+P_{\mathcal{T}_{s}}(y),~\forall~x~,y\in\mathbb{R}^q.
\end{eqnarray}

Define another space at point $s\in\partial \bar\Omega(t)$ as follows:
\begin{eqnarray}
 &&\mathcal{M}_{s}=\{x\in\mathbb{R}^q~\big|x=ky,~y=\lim_{n\rightarrow\infty}\frac{s_n-s}{\|s_n-s\|},~\{s_n\}  \nonumber\\
    &&~~~~~~~~~~~~~~\text{is~a~sequence~belonging~to}~ \bar\Omega(t)\nonumber\\
&&~~~\text{and~satisfying~that}~\lim_{n\rightarrow\infty}s_n=s,~k\in\mathbb{R}\}/\mathcal{T}_{s}.
\end{eqnarray}

When $\dot{s}'(t)\in\mathcal{M}_{s'(t)}$, considering the continuity of $\bar\eta(t)$ with respect to $t$, we can conclude that $\bar\eta(t)\in\partial\bar \Omega(t)$, i.e., $\bar\eta(t)=s_i(t)$. In this case, $\dot{\bar\eta}'(t)=\dot{s}'(t)\in\mathcal{M}_{s'(t)}$. Therefore, there exists an $t_1>t$ such that $\bar\eta(t')\in \Ri\left(\bar \Omega(t)\right),~\forall ~t'\in(t,t_1)$. This implies  $\dot{s}'(t)=\dot{\bar\eta}(t)$, i.e., \eqref{klbn} in Case 2 holds. The remainder of the proof in this case is analogous to that in next case. Therefore, in this case, we assume that $\dot{s}'(t)\notin\mathcal{M}_{s'(t)}$.

When  $\dot{s}'(t)\notin\mathcal{M}_{s'(t)}$, we have
\begin{flalign}
    &P_{\bar\Omega(t)}(\bar\eta(t+\Delta t))=P_{\bar\Omega(t)}\big(\bar\eta(t)+\dot{\bar\eta}(t)\Delta t+\sigma(\Delta t)\big)\nonumber\\
    &=P_{\Omega}\big(\bar\eta(t)\big)+P_{\mathcal{T}_{s'(t)}}\Big(\dot{\bar\eta}(t)\Big)\Delta t+\sigma(\Delta t)
\end{flalign}
where $\sigma(\Delta t)$ is of a higher order than $\Delta t$.
It implies
\begin{eqnarray}
    \dot{s}'(t)=P_{\mathcal{T}_{s'(t)}}\Big(\dot{\bar\eta}(t)\Big).
\end{eqnarray}
Since $\bar\eta(t)- s'(t)\in\mathcal{D}_{s'(t)},$ it follows from \eqref{jkas} and \eqref{kkhad} that
\begin{flalign}
   \dot{s}'(t)=&P_{\mathcal{T}_{s'(t)}}\Big(-\alpha\beta\nabla f(s'(t))+\beta\Delta_s(t)+\beta\Delta_{z}(t)\Big)\nonumber\\
    =&\beta P_{\mathcal{T}_{s'(t)}}\Big(-\alpha\nabla f(s'(t))\Big)+\beta P_{\mathcal{T}_{s'(t)}}\Big(\Delta_s(t)+\Delta_{z}(t)\Big).\nonumber\\
\end{flalign}

Using \eqref{kladk}, we have

\begin{flalign}\label{hak3}
   \frac{dV_1(t)}{dt}=&\beta\bigg\langle  s'(t)-\bar \eta(t)-b(t),  -P_{\mathcal{T}_{s'(t)}}\Big(\alpha\nabla f(s'(t))\Big)\nonumber\\
 &+ P_{\mathcal{T}_{s'(t)}}\Big(\Delta_s(t)+\Delta_{z}(t)\Big)- s'(t)+\bar\eta(t))\nonumber\\
   &+\alpha \nabla f(s'(t)) -\Delta_s(t)-\Delta_{z}(t)-\dot{b}(t)\bigg\rangle\nonumber\\
    \leq&-\beta\| s'(t)-\bar\eta(t)-b\|^2+\beta\Big\langle s'(t)-\bar\eta(t)-b(t)\nonumber\\
   &-\dot{b}(t)+P_{\mathcal{D}_{s'(t)}}\Big(\alpha\nabla f(s'(t))\Big)+b(t)\nonumber\\
    &+P_{\mathcal{T}_{s'(t)}}\Big(\Delta_s(t)+\Delta_{z}(t)\Big)-\Delta_s(t)-\Delta_{z}(t)\Big\rangle.
\end{flalign}

Since $b(t) = y^*(t) - \eta^*(t)$, it follows from \eqref{ojs} that
\begin{equation}
   b(t)= -\alpha\nabla f(y^*((t)))=-P_{\mathcal{D}_{y^*(t)}}\Big(\alpha\nabla f(y^*(t))\Big).
\end{equation}
This implies
\begin{flalign}
   & \Big\|P_{\mathcal{D}_{s'(t)}}\Big(\alpha\nabla f(s'(t))\Big)+b(t)\Big\|\nonumber\\
    =&\Big\|P_{\mathcal{D}_{s'(t)}}\Big(\alpha\nabla f(s'(t))\Big)-P_{\mathcal{D}_{y^*}}\Big(\alpha\nabla f(y^*(t))\Big)\Big\|\nonumber\\
    \leq&\Big\|P_{\mathcal{D}_{s'(t)}}\Big(\alpha\nabla f(s'(t))\Big)-P_{\mathcal{D}_{s'(t)}}\Big(\alpha\nabla f(y^*(t))\Big)\Big\|\nonumber\\
    &+\Big\|P_{\mathcal{D}_{s'(t)}}\Big(\alpha\nabla f(y^*(t))\Big)-P_{\mathcal{D}_{y^*}}\Big(\alpha\nabla f(y^*(t))\Big)\Big\|\nonumber\\
    \leq&\alpha\|\nabla f(s'(t))-\nabla f(y^*(t))\|+\alpha M_1\|\nabla f(y^*(t))\|\|{\eta}'(t)\|\nonumber\\
    \leq&\alpha\Big(L_{max}+\|\nabla f(s^*(t))\|M_1\Big)\|{\eta}'(t)\|
\end{flalign}
where $M_1$ represents an upper bound for the normal curvature of $\partial\bar\Omega(t)$.

Next, we proceed to upper bound $\|\Delta_s(t)\|^2$ and $\|\Delta_z(t)\|^2$.

\begin{flalign}\label{48a}
    \|\Delta_s(t)\|^2=&\|\bar s(t)-s'(t)\|^2\nonumber\\
  =&\left\|P_{\bar\Omega(t)}\left(\sum_{i=1}^N\frac{\eta_i(t)}{N}\right)- \sum_{i=1}^N P_{\Omega}(\eta_i(t))/N\right\|^2\nonumber\\
  \leq&\frac{1}{N}\sum_{i=1}^N\|P_{\bar\Omega(t)}(\bar\eta(t))-  P_{\bar\Omega(t)}(\eta_i(t))\|^2\nonumber\\
  \leq&\frac{1}{N}\sum_{i=1}^N\|\bar\eta(t)- \eta_i(t)\|^2=\frac{1}{N}\|\boldsymbol{\widetilde{\eta}}(t)\|^2.
\end{flalign}

Similarly,
\begin{flalign}\label{49a}
    \|\Delta_z(t)\|^2=&\frac{\alpha^2}{N^2}\left\|  \sum_{i=1}^N \nabla f_i(s'(t))-\sum_{i=1}^N\nabla f_i(y_i(t)) \right\|^2\nonumber\\
 \leq& \frac{\alpha^2L^2_{max}}{N} \sum_{i=1}^N2\left(\|s_i(t)-s'(t)\|^2+\|e_i(t)\|^2\right)\nonumber\\
\leq& \frac{\alpha^2L^2_{max}}{N} \sum_{i=1}^N2\left(\|\eta_i(t)-\bar\eta(t)\|^2+\|e_i(t)\|^2\right)\nonumber\\
=&\frac{2\alpha^2L^2_{max}}{N}\left(\|\boldsymbol{\widetilde\eta}(t)\|^2+\|\boldsymbol{e}(t)\|^2\right).
\end{flalign}

Given that $0_q\in\mathcal{T}_{s'(t)}$, we have
\begin{flalign}
    &\Big\|P_{\mathcal{T}_{s'(t)}}\Big(\Delta_s(t)+\Delta_{z}(t)\Big)-\Delta_s(t)-\Delta_{z}(t)\Big\|\nonumber\\
    \leq&\|\Delta_s(t)+\Delta_{z}(t)\|\leq \|\Delta_s(t)\|+\|\Delta_{z}(t)\|.
\end{flalign}

Furthermore, since $b(t) = s^*(t) - \eta^*(t)$, it follows that
\begin{eqnarray}\label{klaa}
   \|\dot{b}(t)\|\leq\|\dot{\eta}^*(t)\|+\|\dot{s}^*(t)\|\leq 2\|\dot{\eta}^*(t)\|.
\end{eqnarray}

By the Young inequality and \eqref{hak3}, we have
\begin{flalign}\label{haaak3zz}
    \frac{dV_1(t)}{dt}\leq&-\frac{\beta}{3}\| s(t)-\eta(t)-b(t)\|^2\nonumber\\
   &+\beta\alpha^2\Big(L_{max}+\|\nabla f(s^*(t))\|M_1\Big)^2\|{\eta}'(t)\|^2\nonumber\\
    &+(2+4\alpha^2L^2_{max})\beta
    \|\boldsymbol{\widetilde{\eta}}(t)\|^2/N\nonumber\\
   &+\frac{4\beta\alpha^2L^2_{max}}{N}\|\boldsymbol{e}(t)\|^2+6\beta\|\dot{\eta}^*(t)\|^2.
\end{flalign}

Using \eqref{jkas}, we obtain
\begin{flalign}\label{havaak3}
  \frac{dV_2(t)}{dt} =& \beta\Big\langle  \eta'(t),s'(t)-\bar\eta(t)-\alpha \nabla f(s'(t))\nonumber\\
   & +\Delta_s(t)+\Delta_{z}(t)\Big\rangle+\left\langle \eta'(t), \dot{\eta}^*(t)\right\rangle\nonumber\\
    = &\beta\Big\langle  \eta'(t),s'(t)-\bar\eta(t)-\alpha \nabla f(s'(t))-y^*(t)\nonumber\\
    & +\eta^*(t)+\alpha\nabla f(y^*(t))+\Delta_s(t)+\Delta_{z}(t)+\dot{\eta}^*(t)\Big\rangle\nonumber\\
   = &\beta\Big\langle \eta'(t),s'(t)-{\eta}'(t)-\alpha \nabla f(s'(t))-y^*(t)\nonumber\\
    & +\alpha\nabla f(s^*(t))+\Delta_s(t)+\Delta_{z}(t)+\dot{\eta}^*(t)\Big\rangle.
\end{flalign}

According to the projection inequality, $\|s'(t)-y^*(t)\|^2\leq\|\eta'(t)\|^2$. This enables us to deduce from \eqref{48a} and \eqref{49a} that
\begin{flalign}\label{hqqaaak3}
 & \langle \eta'(t),s'(t)-\eta'(t)-y^*(t)+\Delta_s(t)+\Delta_{z}(t)\rangle\nonumber\\
     = &-\|\eta'(t)\|^2+\big\langle \eta'(t),s'(t)-y^*(t)+\Delta_s(t)+\Delta_{z}(t)+\dot{\eta}^*(t)\big\rangle\nonumber\\
     \leq& -\frac{1}{2}\|{\eta}'(t)\|^2+\frac{1}{2}\|s'(t)-y^*(t)\|\nonumber\\
    &+\big\langle \eta'(t),\Delta_s(t)+\Delta_{z(t)}+\dot{\eta}^*(t)\big\rangle\nonumber\\
     \leq &\frac{\sigma\alpha}{4}\|{\eta}'(t)\|^2+\frac{3}{\sigma\alpha}\left(\|\Delta_s(t)\|^2+\|\Delta_{z(t)}\|^2+\|\dot{\eta}^*(t)\|^2\right)\nonumber\\
    \leq &\frac{\sigma\alpha}{4}\|{\eta}'(t)\|^2+\frac{3}{\sigma\alpha}\Bigg(\frac{2\alpha^2L^2_{max}+1}{N}\|\boldsymbol{\widetilde\eta}(t)\|^2\nonumber\\
    &+\frac{2\alpha^2L^2_{max}}{N}\|\boldsymbol{e}(t)\|^2+\|\dot{\eta}^*(t)\|^2\Bigg).
\end{flalign}
Since
\begin{eqnarray}\label{hqqaaak3aa}
   && n'(t)=\bar\eta(t)-n^*(t)=\bar\eta(t)+b(t)-s^*(t)
\end{eqnarray}
we have
\begin{flalign}\label{hqqaaaka3}
 &\langle \eta'(t),-\alpha \nabla f(s'(t))+\alpha \nabla f(y^*(t))\rangle\nonumber\\
    =&\big\langle \bar\eta(t)-s'(t)+b(t),-\alpha \nabla f(s'(t))+\alpha \nabla f(y^*(t))\big\rangle\nonumber\\
    &+\big\langle s'(t)-s^*(t),-\alpha \nabla f(s'(t))+\alpha \nabla f(y^*(t))\big\rangle.
\end{flalign}

With the application of Young inequality,
\begin{flalign}
& \langle \bar\eta(t)-s'(t)+b(t),-\alpha \nabla f(s'(t))+\alpha \nabla f(y^*(t))\rangle\nonumber\\
    \leq &\frac{1}{6}\|s'(t)-\bar\eta(t)-b(t)\|^2+3\alpha^2\|\nabla f(s'(t))- \nabla f(y^*(t))\|^2\nonumber\\
    \leq& \frac{1}{6}\|s'(t)-\bar\eta(t)-b(t)\|^2+3\alpha^2L_{max}^2\|s'(t)-y^*(t)\|^2\nonumber\\
   \leq& \frac{1}{6}\|s'(t)-\bar\eta(t)-b(t)\|^2+3\alpha^2L_{max}^2\|\eta'(t)\|^2.
\end{flalign}

For any $a,~b\in\mathbb{R}^p,$
\begin{eqnarray}\label{hkda}
    -\|a-b\|^2= \!-\|a\|^2+2a^Tb-\|b\|^2\leq\!-\frac{\|a\|^2}{2}\!+\!\|b\|^2.
\end{eqnarray}

Recalling that the objective function $f$ is $\sigma$-strongly convex, where $\sigma=\sum_{i=1}^N\sigma_i/N,$
it follows from \eqref{hkda} that
\begin{flalign}
  & \Big\langle s'(t)-y^*(t),-\alpha \nabla f(s'(t))+\alpha \nabla f(y^*(t))\Big\rangle\nonumber\\
    \leq& -\sigma\alpha \|s'(t)-y^*(t)\|^2\nonumber\\
   = &-\sigma\alpha \left\|\bar\eta(t)-\eta^*(t)-\left(\bar\eta(t)-s'(t)+b(t)\right)\right\|^2\nonumber\\
\leq& -\frac{\sigma\alpha}{2}\|\bar\eta(t)-\eta^*(t)\|^2+\sigma\alpha\|\bar\eta(t)-s'(t)+b(t)\|^2.
\end{flalign}

Substituting \eqref{hqqaaak3} and \eqref{hqqaaaka3} into \eqref{havaak3} yields
\begin{flalign}\label{klassa}
\frac{dV_2(t)}{dt}\leq& -\beta\left(\frac{\sigma\alpha}{4}-3\alpha^2L_{max}^2\right)\|\bar\eta(t)-\eta^*(t)\|^2\nonumber\\
    &+\beta\left(\sigma\alpha+\frac{1}{6}\right)\|\bar\eta(t)-s'(t)-b\|^2\nonumber\\
    &+\frac{3\beta}{\sigma\alpha}\Bigg(\frac{2\alpha^2L^2_{max}+1}{N}\|\boldsymbol{\widetilde\eta}(t)\|^2\nonumber\\
    &+\frac{2\alpha^2L^2_{max}}{N}\|\boldsymbol{e}(t)\|^2+\|\dot{\eta}^*(t)\|^2\Bigg).
\end{flalign}

By combining \eqref{haaak3zz} and \eqref{klassa}, we conclude
\begin{flalign}\label{lghje}
 \frac{dV_{\eta'}(t)}{dt} \leq & -\left[\frac{\sigma\alpha}{4}-\alpha^2\left(L_{max}+\|\nabla f(y^*(t))\|M_1\right)^2\right.\nonumber\\
    & -3\alpha^2L_{max}^2~\Big]\beta\|\bar\eta(t)-\eta^*\|^2\nonumber\\
    & -\left(\frac{1}{6}-\sigma\alpha\right)\beta\|\bar\eta(t)-s'(t)-b(t)\|^2\nonumber\\
    & +\beta\left(\frac{6\alpha L^2_{max}+3}{N\sigma}+\frac{2+4\alpha^2L_{max}^2}{N}\right)\|\boldsymbol{\widetilde\eta}(t)\|^2\nonumber\\
    & +\beta\left(\frac{6\alpha L^2_{max}}{\sigma N}+\frac{4\alpha^2L^2_{max}}{N}\right)\|\boldsymbol{e}(t)\|^2\nonumber\\
    & \left. +\beta\left(\frac{3}{\sigma\alpha}+6\right)\|\dot{\eta}^*(t)\|^2\right..
\end{flalign}

\textbf{Case 2: $s'(t)\in \int(\bar\Omega(t))$.}

In this case, $\bar\eta(t)\in \int(\bar\Omega(t))$. This implies  $P_{\Omega}(\bar\eta(t+\Delta t))=\bar\eta(t+\Delta t)$ and $P_{\Omega}(\bar\eta(t))=\bar\eta(t)$.
Hence,
\begin{equation}\label{klbn}
    \dot{s}'(t)=\lim_{\Delta t\rightarrow0}\frac{P_{\Omega}(\bar\eta(t+\Delta t))-P_{\Omega}(\bar\eta(t))}{\Delta t}=\dot{\bar\eta}(t).
\end{equation}

From \eqref{klaa}, it follows that
\begin{flalign}
 \frac{dV_1(t)}{dt}&=\beta\langle b(t), \dot{b}(t) \rangle\leq\frac{\beta}{8}\|b(t)\|^2+2\beta\|\dot{b}(t)\|^2\nonumber\\
  &\leq\frac{\beta}{8}\|b(t)\|^2+4\beta\|\dot{\eta}^*(t)\|^2.
\end{flalign}

The derivative of $V_2(t)$ can be described as follows:
\begin{flalign}\label{haaak3}
 \frac{dV_2(t)}{dt}=&\beta\langle {\eta}'(t),-\alpha \nabla f(\bar\eta(t))+\Delta_s(t)+\Delta_{z}(t)-\dot{\eta}^*(t)\rangle\nonumber\\
   =&\beta\langle \eta'(t),-\alpha \nabla f(\bar\eta(t))+b(t)\rangle-\beta\langle {\eta}'(t), b(t)\rangle\nonumber\\
   &+\beta\langle {\eta}'(t),\Delta_s(t)+\Delta_{z}(t)-\dot{\eta}^*(t)\rangle.
\end{flalign}

Since $\bar\Omega(t)$ is a convex set and $\bar\eta(t)\in\bar\Omega(t)$, we have
\begin{flalign}\label{uiape}
& \langle\bar\eta(t)-s^*(t), b(t)\rangle \nonumber \\
   =&\langle\bar\eta(t)-P_{\Omega}(\eta^*(t)),P_{\bar\Omega(t)}(\eta^*(t))-\eta^*(t)\rangle\geq0.
\end{flalign}
Hence,
\begin{flalign}\label{jladad}
\langle {\eta}'(t), b(t)\rangle=& \langle\bar\eta(t)-\eta^*(t), b(t)\rangle\nonumber\\
= & \langle\bar\eta(t)-y^*(t)+b(t), b(t)\rangle\geq \|b(t)\|^2.
\end{flalign}

Using the Young inequality, we have
\begin{flalign}\label{kafs1}
 &\langle b(t),-\alpha \nabla f(\bar\eta(t))+b(t)\rangle\nonumber\\
  = &\langle b(t),-\alpha \nabla f(\bar\eta(t))+\alpha \nabla f(y^*(t))\rangle\nonumber\\
 \leq&\frac{1}{2}\|b(t)\|^2+\frac{1}{2}\alpha^2L_{max}^2\|\bar\eta(t)-y^*(t)\|^2.
\end{flalign}

Considering the strong convexity of the objective function $f$, we get
\begin{flalign}\label{kafs2}
 &\langle s'(t)-y^*(t),-\alpha \nabla f(\bar\eta(t))+b(t)\rangle\nonumber\\
  =&\langle \bar\eta(t)-y^*(t),-\alpha \nabla f(\bar\eta(t))+\alpha \nabla f(y^*(t))\rangle\nonumber\\
 \leq&-\sigma\|\bar\eta(t)-y^*(t)\|^2.
\end{flalign}

Given that $s'(t) \in \int(\bar\Omega(t))$,  $s'(t) = \bar{\eta}(t)$. It implies
\begin{flalign}
  {\eta}'(t) &= \bar{\eta}(t) - \eta^*(t) = s'(t) - \eta^*(t)\nonumber \\
               &= s'(t) - y^*(t) + b(t).
\end{flalign}

By utilizing \eqref{kafs1} and \eqref{kafs2}, we obtain
\begin{flalign}\label{kafsaad}
&\langle \eta'(t), -\alpha \nabla f(\bar{\eta}(t)) + b(t) \rangle \nonumber\\
    = &\langle s'(t) - y^*(t) + b(t), -\alpha \nabla f(\bar{\eta}(t)) + b(t) \rangle \nonumber\\
    \leq& \frac{1}{2}\|b(t)\|^2 + \frac{1}{2}\alpha^2 L_{\text{max}}^2 \|\bar{\eta}(t) - y^*(t)\|^2 \nonumber\\
    & - \sigma \|\bar{\eta}(t) - y^*(t)\|^2.
\end{flalign}

Substituting equations \eqref{jladad} and \eqref{kafsaad} into equation \eqref{haaak3} results
\begin{flalign}\label{haaak3aa}
 \frac{dV_2(t)}{dt} =& \beta\langle {\eta}'(t),-\alpha \nabla f(\eta(t))\rangle \nonumber\\
  =&-\beta\left(\sigma-\frac{1}{2}\alpha^2L_{\max}^2\right)\|\bar\eta(t)-y^*(t)\|^2-\frac{\beta}{2}\|b(t)\|^2 \nonumber\\
  &+\beta\langle {\eta}'(t),\Delta_s(t)+\Delta_{z}(t)+\dot\eta^*(t)\rangle.
\end{flalign}

Utilizing inequality \eqref{hkda}, we have
 \begin{flalign}
  &\|\bar\eta(t)-s^*(t)\|^2 = \|\bar\eta(t)-\eta^*(t)+\eta^*(t)-y^*(t)\|^2 \nonumber\\
   \geq& \frac{1}{2}\|\bar\eta(t)-\eta^*(t)\|^2 - 2\|\eta^*(t)-y^*(t)\|^2 \nonumber\\
   \geq& \frac{1}{2}\|\bar\eta(t)-\eta^*(t)\|^2 - 2\|b(t)\|^2.
\end{flalign}

By \eqref{haaak3aa}, we have
\begin{flalign}\label{haaak3aaaa}
\frac{dV_2(t)}{dt} \leq& \beta\langle {\eta}'(t),\Delta_s(t)+\Delta_{z}(t)+\dot\eta^*(t)\rangle\nonumber\\
&-\frac{\beta}{4}(\sigma_1-\alpha^2L_{\max}^2)\|\eta(t)-\eta^*(t)\|^2\nonumber \\
& -\frac{\beta}{4}\|b(t)\|^2 .
\end{flalign}
where $\sigma_1=\min\{\sigma,\frac{1}{16}\}$.

Applying the Young inequality, we have

\begin{flalign}
&\langle {\eta}'(t),\Delta_s(t)+\Delta_{z}(t)+\dot\eta^*(t)\rangle\nonumber \\
 \leq& \frac{1}{4}\alpha^2L_{\max}^2\|{\eta}'(t)\|+\frac{2}{\alpha^2L_{\max}^2}\|\Delta_s(t)+\Delta_{z}(t)+\dot\eta^*(t)\|^2.
\end{flalign}

 It implies
\begin{flalign}\label{khbf}
\frac{dV_{\eta'}(t)}{dt}=&\frac{dV_2(t)}{dt}\nonumber\\
\leq&-\frac{\beta}{2}(\sigma_1-2\alpha^2L_{\max}^2)V_2(t)-\frac{\beta}{4}V_1(t)\nonumber\\
&+\beta\left(\frac{6}{\alpha^2L_{\max}^2N^2}+\frac{12}{N^2}\right)\|\boldsymbol{\widetilde{\eta}}(t)\|^2\nonumber\\
&+\frac{12\beta}{N^2}\|\boldsymbol{e}(t)\|^2+\left(\frac{6}{\alpha^2L_{\max}^2N}+4\right)\beta\|\dot\eta^*(t)\|^2.
\end{flalign}

By combining equations \eqref{lghje} and \eqref{khbf}, we conclude that equation \eqref{khakw} always holds.
Adding inequalities \eqref{asfaaasd1}, \eqref{aklew}, and \eqref{khakw}, we have that \eqref{jkasc} is always satisfied.\end{proof}


\textbf{Proof~of~Inequality~\eqref{sma3}:}
Recalling $\widetilde{x}_{i}(t) = {x}_{i}(t)-\Pi_is_i(t)$, we have
\begin{flalign}\label{closed-loop sys 22}
\dot{\widetilde{x}}_{i}(t)=&\dot{x}_{i}(t)-\Pi_i \dot{s}_i(t)\nonumber\\
=& (A_i+B_iK_{1i})x_{i}(t)+B_iK_{2i}s_i(t)-\Pi_i \dot{s}_i(t)\nonumber\\
=&(A_i+B_iK_{1i})\widetilde{x}_{i}(t)-\Pi_i \dot{s}_i(t).
\end{flalign}

From \eqref{closed-loop sys 22},
\begin{flalign}\label{kaaasd}
 \|\dot{e}_i(t)\| =&\|C_i\dot{\widetilde{x}}_{i}(t)\|\nonumber\\
  \leq&\|C_i(A_i+B_iK_{1i})\widetilde{x}_{i}(t)\|+\|C_i\Pi_i\dot{s}_i(t)\|.
  \end{flalign}

Similarly to \eqref{uhww}, we have
\begin{flalign}\label{uhwkw}
\|\dot{\boldsymbol{\eta}}(t)\|\leq (\lambda_N k_1+2)\beta \|\boldsymbol{\widehat\eta}(t)\|+\alpha\beta\|\boldsymbol{\widehat z}(t)\|.
\end{flalign}

Considering that $\|\dot{s}_i(t)\|\leq\|\dot{\eta}_i(t)\|$, it follows from \eqref{uhwkw} and \eqref{kaaasd} that
\begin{flalign}\label{kaaassad}
 \|\dot{\boldsymbol{e}}(t)\|\leq&\max_{i\in\mathcal{V}}\|C_i(A_i+B_iK_{1i})\|\|\boldsymbol{\widetilde{x}}(t)\|\nonumber\\
&+\max_{i\in\mathcal{V}}\|C_i\Pi_i\| (\lambda_N k_1+2)\beta \|\boldsymbol{\widehat\eta}(t)\|\nonumber\\
 &+\max_{i\in\mathcal{V}}\|C_i\Pi_i\|\alpha\beta\|\boldsymbol{\widehat z}(t)\|.
  \end{flalign}
By \eqref{jklaw},
\begin{flalign}\label{kaaassads}
 \|\dot{\boldsymbol{e}}(t)\|\leq&\max_{i\in\mathcal{V}}\|C_i(A_i+B_iK_{1i})\|\|\boldsymbol{\widetilde{x}}(t)\|\nonumber\\
&+\max_{i\in\mathcal{V}}\|C_i\Pi_i\| (\lambda_N k_1+2+L_{max}\alpha)\beta \|\widehat\eta(t)\|\nonumber\\
& \!+\max_{i\in\mathcal{V}}\|C_i\Pi_i\|\alpha\beta\left(\|\boldsymbol{\widetilde z}(t)\|+L_{max}\|\boldsymbol{e}(t)\|\right).
  \end{flalign}

 Since $\widehat\eta_i(t)=\widetilde{\eta}_i(t)+\eta'(t),$
\begin{flalign}\label{kaaassad}
 \|\boldsymbol{\widehat\eta}(t)=&\|\boldsymbol{\widetilde{\eta}}(t)+\boldsymbol{1}_N\otimes \eta'(t)\|\nonumber\\
\leq &\|\boldsymbol{\widetilde{\eta}}(t)\|+\sqrt{N}\|\eta'(t)\|\leq\sqrt{2}\|\boldsymbol{\theta}(t)\|.
 \end{flalign}

Additionally, as $\nabla f_i$ is $L_i$-Lipschitz, we have $\|\dot{z}^*(t)\|\leq L_{max}\|\dot{y}^*(t)\|$. Furthermore, based on \eqref{hqqaaak3aa}, we can conclude that $\|\dot{\eta}^*(t)\|\leq (L_{max}+1)\|\dot{y}^*(t)\|$.

Then we can rewrite \eqref{jkasc} as follows.
\begin{flalign}\label{klasas1}
\frac{d V_{\boldsymbol\theta}(t)}{dt} &\leq -\beta(\gamma_{\eta}-\beta\gamma_{\theta}')V_{\boldsymbol\theta}(t) + \beta\gamma^{\eta^*}_{\eta}(L_{max}+1)\|\dot{y}^*(t)\|^2 \nonumber \\
&\quad + \beta\gamma_{\theta}^{\tilde x}\|\boldsymbol{\widetilde{x}}(t)\|^2.\end{flalign}

Given $\beta\leq \frac{\gamma_{\eta}}{2\gamma_{\theta}'},$ it follows from \eqref{klasas1} that
\eqref{sma3} holds.

 \section{{Proof~of~Inequality~\eqref{sma2}}}\label{2ap2}
 Let $V_{\boldsymbol{\widetilde{x}}}(\boldsymbol{\widetilde{x}},t)=\|\boldsymbol{\widetilde{x}}(t)\|^2.$
 By \eqref{closed-loop sys 22} and Young inequality, we have
\begin{flalign}\label{kaacassd}
\frac{d V_{\boldsymbol{\widetilde{x}}}(\boldsymbol{\widetilde{x}},t)}{dt} &\leq -cV_{\boldsymbol{\widetilde{x}}}(\boldsymbol{\widetilde{x}},t)\nonumber\\
&\quad -2\sum_{i=1}^N\Big\langle P_i({x}_{i}(t)-\Pi_is_i(t)),~\Pi_i\dot{s}_i(t)\Big\rangle \nonumber\\
&\leq -\frac{c}{2}V_{\boldsymbol{\widetilde{x}}}(\boldsymbol{\widetilde{x}},t)+\sum_{i=1}^N\frac{2\|P_i\|\|\Pi_i\|^2}{c}\|\dot{s}_i(t)\|.
\end{flalign}

As a result,
\begin{flalign}\label{kaloaaasd}
  \|\boldsymbol{e}(t)\|^2=\sum_{i=1}^N\|C_i\widetilde{x}_{i}(t)\|^2\leq\frac{C^2_{max}}{\lambda_{min}(P_i)}V_{x}(x,t)
  \end{flalign}
where $C_{max}:=\max_{i\in\mathcal{V}}\|C_i\|$ and $\lambda_{min}(P_i)$ denotes the smallest eigenvalue of the matrix $P_i.$
 Substituting  \eqref{klasacas} and \eqref{uhawwqada} into \eqref{kaacassd} yields
\begin{flalign}\label{kasasdd}
\frac{d V_{\boldsymbol{\widetilde{x}}}(\boldsymbol{\widetilde{x}},t)}{dt} \leq& -\frac{c}{2}V_{\boldsymbol{\widetilde{x}}}(\boldsymbol{\widetilde{x}},t)+\sum_{i=1}^N\frac{2\|P_i\|\|\Pi_i\|^2}{c}\Big((3\lambda_N^2 k_1^2\nonumber\\
& +12+9\alpha^2L^2_{max})\beta^2 \|\widehat\eta_i(t)\|^2+9\beta^2\alpha^2\|\widetilde z_i(t)\|^2\nonumber\\
 &+\frac{9\beta^2\alpha^2L^2_{max}C^2_{max}}{\lambda_{min}(P_i)}\|e_i(t)\|^2\Big).
  \end{flalign}

  Substituting  \eqref{kaloaaasd} into \eqref{kasasdd}, we have
\begin{flalign}\label{hkag1}
\frac{d V_{\boldsymbol{\widetilde{x}}}(\boldsymbol{\widetilde{x}},t)}{dt} \leq&- \left(\frac{c}{2}-\gamma_{x}\beta^2\right)V_{\boldsymbol{\widetilde{x}}}(\boldsymbol{\widetilde{x}},t) \nonumber\\
&+ \beta^2\gamma^{\widehat\eta}_{x}\|\boldsymbol{\widehat\eta}(t)\|^2+\beta^2\gamma^{\widetilde z}_{x}\|\boldsymbol{\widetilde z}(t)\|^2.
\end{flalign}

Given $\beta\leq\sqrt{\frac{c}{4\gamma_x}}$ and \eqref{hkag1}, we have \eqref{sma2} holds.

 \section{{Proof~of~Lemma~\ref{lem523}}}\label{2ap4}

Given $x_i(t)\in\mathcal{R}_{s_i}(t)$, we have
\begin{eqnarray}
    \|x_i(t)-\Pi_is_i(t)\|\leq \frac{r(t)}{\|C_i\|}.
\end{eqnarray}
where $\Pi_i$ is defined the same as in  Lemma~\ref{lem31}.

By using the second equation of \eqref{ch4ass7}, the following inequality always holds:
\begin{equation}
\begin{aligned}
\|y_i(t)-s_i(t)\|&=\|C_ix_i(t)-s_i(t)\|\\
&=\|C_ix_i(t)-C_i\Pi_is_i(t)\|\\
&\leq\|C_i\|\|x_i(t)-\Pi_is_i(t)\|\leq r(t).\label{111}
\end{aligned}
\end{equation}
Since $\mathcal{S}_{s_i,\bar{\Omega}}(t)$
is the $p$-dimension open ball with $s_{i}(t)$ as the center and $r(t)$ as the radius from \eqref{uhas}, we directly arrive at $y_{i}\in\mathcal{S}_{s_i,\bar{\Omega}}(t)$ from~\eqref{111}.

Given $\widetilde{x}_{i}(t) = {x}_{i}(t)-\Pi_i\eta_i(t)$ and $y_i(0)=\eta_i(0)$,  we have
\begin{eqnarray}
   \widetilde{x}_{i}(0) \!\!\!\!&=&\!\!\!\!{x}_{i}(0)-\Pi_iy_i(0)\nonumber\\
   \!\!\!\!&=&\!\!\!\!{x}_{i}(0)-\Pi_iC_ix_i(0).
\end{eqnarray}

Since $\Pi_i=\left[
        \begin{array}{c}
         I_p\\
         \bold{0_{(n_i-p )\times p }}\\
        \end{array}
      \right]$,
 \begin{eqnarray}\label{jkaaa}
\Pi_iC_i=\left[
        \begin{array}{cc}
         I_p&\bold{0_{p\times{(n_i-p)}}}\\
       \bold{0_{{(n_i-p)}\times p}}&\bold{0_{(n_i-p )\times(n_i-p )}}\\
        \end{array}
      \right]
\end{eqnarray}

According to  Assumption \ref{as42}, i.e.,  $x_{ij}(0)=0,~p<j\leq n_i$, we have $\widetilde{x}_{i}(0) =\bold{0_{n_i}}$. Hence, $x_i(0)\in\mathcal{R}_i(0)$.

 \section{{Proof~of~Lemma~\ref{smaas}}}\label{2ap5}

Since $\sqrt{\frac{cH_1}{8\beta_1\xi M'}}$ and $v\leq\frac{c}{8}$,
\begin{eqnarray}\label{opas}
   \dot{\beta}_2(t)\geq-\frac{c}{8}\beta_2(t),~ \frac{c}{8}\beta_2(t)\geq \beta^2M'V_{\boldsymbol\theta}(t).
\end{eqnarray}

Next, let us analyze the time derivative of $ h_i(x_i,s_i,t)$,
\begin{flalign}\label{96as}
  \frac{d h_i(x_i,s_i,t)}{dt}=&\frac{\partial h_i(x_i,s_i,t)}{\partial{x}^T_i(t)}(A_ix_i(t)+B_iu_i(t))\nonumber\\
&+\frac{\partial h_i(x_i,s_i,t)}{\partial{s}^T_i(t)}\dot{s}_i(t)+\frac{\partial h_i(x_i,s_i,t)}{\partial{t}}.
  \end{flalign}
Consider the first two terms in \eqref{96as},
\begin{flalign}
&\frac{\partial h_i(x_i(t),t)}{\partial{x}^T_i(t)}\left(A_ix_i(t)+B_iu_i(t)\right)+\frac{\partial h_i(x_i,s_i,t)}{\partial{s}^T_i(t)}\dot{s}_i(t)\nonumber\\
=&-2\left\langle P_i\left({x}_{i}(t)-\Pi_is_i(t)\right),~A_ix_i(t)+B_iu_i(t)\right\rangle\nonumber\\
&+2\left\langle P_i\left({x}_{i}(t)-\Pi_is_i(t)\right),~\Pi_i\dot{s}_i(t)\right\rangle\nonumber\\
=&-2\left\langle P_i\widetilde{x}_i(t),~(A_i+B_iK_{1i})\widetilde{x}_{i}(t)-\Pi_i \dot{s}_i(t)\right\rangle.
  \end{flalign}

Using the symmetry of matrix $P_i$ and the Lyapunov equation \eqref{ria12}, we have
\begin{flalign}\label{jkda}
&-2\left\langle P_i\widetilde{x}_i(t),~(A_i+B_iK_{1i})\widetilde{x}_{i}(t)\right\rangle\nonumber\\
=&-\widetilde{x}^T_{i}(t)((A_i+B_iK_{1i})^TP_i+P_i(A_i+B_iK_{1i}))\widetilde{x}_{i}(t)\nonumber\\
=&c\widetilde x_i(t)^TP_i\widetilde x_i(t).
  \end{flalign}
Clearly,
\begin{eqnarray}\label{klasacas}
 \|\dot{s}_i(t)\|\!\!\!\!&=& \!\!\!\! \lim_{\Delta t\rightarrow0}\frac{\|P_{\bar\Omega(t)}(\eta_i(t+\Delta t))-P_{\bar\Omega(t)}(\eta_i(t))\|}{\Delta t}\nonumber\\
 \!\!\!\!&\leq& \!\!\!\! \lim_{\Delta t\rightarrow0}\frac{\|\eta_i(t+\Delta t)-\eta_i(t)\|}{\Delta t}=\|\dot{\eta}_i(t)\|.
\end{eqnarray}
Utilizing the Young inequality, we have
\begin{eqnarray}\label{132}
 &&~~~~2\left\langle P_i\widetilde{x}_i(t),\Pi_i \dot{s}_i(t)\right\rangle\nonumber\\
&& \geq-\frac{c}{2}\widetilde x_i(t)^TP_i\widetilde x_i(t)-\frac{2}{c}\|P_i\|\|\Pi_i\|^2\| \dot{s}_i(t)\|^2\nonumber\\
&& \geq-\frac{c}{2}\widetilde x_i(t)^TP_i\widetilde x_i(t)\!-\!\frac{2}{c}\|P_i\|\|\Pi_i\|^2\| \dot{\eta}_i(t)\|^2.
\end{eqnarray}

  Substituting \eqref{jkda}, \eqref{132} into \eqref{96as} yields
\begin{align}
\frac{d h_i(x_i,s_i,t)}{dt} \geq&~ \dot{\beta}_2(t) + \frac{c\widetilde x_i(t)^TP_i\widetilde x_i(t)}{2} \nonumber \\
&- \frac{2}{c}\|P_i\|\|\Pi_i\|^2\| \dot{\eta}_i(t)\|^2.
\end{align}

Combining \eqref{uhww} and \eqref{jklaw}, we have
\begin{align}\label{uhawwqada}
\|\dot{\boldsymbol{\eta}}(t)\|^2 \leq&~ (3\lambda_N^2 K_1^2+12+9\alpha^2L^2_{max})\beta^2 \|\boldsymbol{\widehat\eta}(t)\|^2 \nonumber\\
&+9\beta^2\alpha^2\|\boldsymbol{\widetilde z}(t)\|^2+9\beta^2\alpha^2L^2_{max}\|\boldsymbol{e}(t)\|^2.
\end{align}

It implies
\begin{align}
  \frac{d h(\boldsymbol x,\boldsymbol s,t)}{dt} \geq &\sum_{i=1}^N\frac{c\widetilde x_i(t)^TP_i\widetilde x_i(t)}{2} - \beta^2M'V_{\boldsymbol\theta}(t) \nonumber\\
  &-\frac{18\beta^2}{c}\alpha^2L^2_{max}\|P_i\|\|\Pi_i\|^2\|\|\boldsymbol{e}(t)\|^2.
\end{align}

Since
\begin{equation}\label{uhwwq}
\begin{aligned}
y_i(t)-\eta_i(t)=C_ix_i(t)-C_i\Pi_i\eta_i(t)=C_i{\widetilde{x}}_{i}(t)
\end{aligned}
\end{equation}
we have
\begin{equation}\label{whww}
\begin{aligned}
\|\boldsymbol{e}(t)\|^2=\sum_{i=1}^N\|y_i(t)-\eta_i(t)\|^2\leq C^2_{max}\|\boldsymbol{\widetilde{x}}(t)\|^2.
\end{aligned}
\end{equation}


Denote $V_{\boldsymbol\theta,\boldsymbol{\widetilde{x}}}(t):=V_{\boldsymbol\theta}(t)+V_{\boldsymbol{\widetilde{x}}}(\boldsymbol{\widetilde{x}},t)$. Since $\widetilde{x}_{i}(0) =\bold{0_{n_i}}$, $V_{\boldsymbol\theta,\boldsymbol{\widetilde{x}}}(0)=\|\boldsymbol\theta(0)\|^2/2.$

Since $\beta\leq\min\{\frac{c}{8\gamma_{\theta}^{\tilde x}},\frac{\gamma_{\theta}}{8\max\{\gamma^{\widehat\eta}_{x},2\gamma^{\boldsymbol{\widetilde z}}_{x}\}},\frac{c}{2\gamma_{\eta}}\}$ and $v=\frac{\beta\gamma_{\eta}}{4}$, it follows from  \eqref{klasas1} and \eqref{hkag1} that
\begin{align}\label{klasassa}
\frac{d V_{\boldsymbol\theta,\boldsymbol{\widetilde{x}}}(t)}{dt} &\leq -vV_{\boldsymbol\theta,\boldsymbol{\widetilde{x}}}(t) + \beta\gamma^{\eta^*}_{\eta}(L_{max}+1)\|\dot{y}^*(t)\|^2 .\end{align}

Given $\|\dot{y}^*(t)\|\leq h_2 e^{-vt}$,
\begin{align}\label{hkag}
V_{\boldsymbol\theta,\boldsymbol{\widetilde{x}}}(t)\leq \left[V_{\boldsymbol\theta,\boldsymbol{\widetilde{x}}}(0)+\frac{\beta\gamma^{\eta^*}_{\eta}(L_{max}+1)h^2}{v}\right]e^{-vt}.
\end{align}

Since $V_{\boldsymbol\theta,\boldsymbol{\widetilde{x}}}(0)=\|\boldsymbol\theta(0)\|^2/2$,
\begin{eqnarray}\label{jkdsasa}
    V_{\boldsymbol\theta}(t)\leq H_1e^{-vt}.
\end{eqnarray}

According to the proof of Theorem \ref{the31}, $\|\widehat\eta(t)\|^2+\|\widetilde z(t)\|^2\leq 2V_{\boldsymbol\theta}(t)$.
By \eqref{jkdsasa},
\begin{align}
  \frac{d h(\boldsymbol x,\boldsymbol s,t)}{dt} \geq&~ \frac{c}{2}\sum_{i=1}^N\widetilde x_i(t)^TP_i\widetilde x_i(t)-\beta^2M'V_{\boldsymbol\theta}(t) \nonumber\\
  &+\dot{\beta}(t)-\frac{18\beta^2}{c}\alpha^2L^2_{max}\kappa_{max}(P) \nonumber\\
  &\times \Pi_{max}^2C^2_{max}\sum_{i=1}^N\widetilde x_i(t)^TP_i\widetilde x_i(t) \nonumber\\
  \geq& \frac{c}{4}\sum_{i=1}^N\widetilde x_i(t)^TP_i\widetilde x_i(t)-\beta^2M'V_{\boldsymbol\theta}(t)+\dot{\beta}_2(t).
\end{align}

Since $\beta(t)$ satisfies \eqref{opas},
\begin{eqnarray*}
   \frac{d h(\boldsymbol x,\boldsymbol s,t)}{dt}+\frac{c}{8}h(\boldsymbol x,\boldsymbol s,t)\geq0.
\end{eqnarray*}
Hence, \eqref{cbf9} is a time-varying CBF.


\bibliographystyle{IEEEtran}
\bibliography{references}

\begin{thebibliography}{10}
\providecommand{\url}[1]{#1}
\csname url@rmstyle\endcsname
\providecommand{\newblock}{\relax}
\providecommand{\bibinfo}[2]{#2}
\providecommand\BIBentrySTDinterwordspacing{\spaceskip=0pt\relax}
\providecommand\BIBentryALTinterwordstretchfactor{4}
\providecommand\BIBentryALTinterwordspacing{\spaceskip=\fontdimen2\font plus
\BIBentryALTinterwordstretchfactor\fontdimen3\font minus
  \fontdimen4\font\relax}
\providecommand\BIBforeignlanguage[2]{{%
\expandafter\ifx\csname l@#1\endcsname\relax
\typeout{** WARNING: IEEEtran.bst: No hyphenation pattern has been}%
\typeout{** loaded for the language `#1'. Using the pattern for}%
\typeout{** the default language instead.}%
\else
\language=\csname l@#1\endcsname
\fi
#2}}

\bibitem{he2018distributed}
X.~He, J.~Yu, T.~Huang, and C.~Li, ``Distributed power management for dynamic
  economic dispatch in the multimicrogrids environment,'' \emph{IEEE
  Transactions on Control Systems Technology}, vol.~27, no.~4, pp. 1651--1658,
  2018.

\bibitem{doan2020distributed}
T.~T. Doan and C.~L. Beck, ``Distributed resource allocation over dynamic
  networks with uncertainty,'' \emph{IEEE Transactions on Automatic Control},
  vol.~66, no.~9, pp. 4378--4384, 2020.

\bibitem{chen2019minimum}
F.~Chen and J.~Chen, ``Minimum-energy distributed consensus control of
  multiagent systems: A network approximation approach,'' \emph{IEEE
  Transactions on Automatic Control}, vol.~65, no.~3, pp. 1144--1159, 2019.

\bibitem{li2022survey}
X.~Li, L.~Xie, and N.~Li, ``A survey of decentralized online learning,''
  \emph{arXiv preprint arXiv:2205.00473}, 2022.

\bibitem{dai2019distributed}
P.~Dai, W.~Yu, G.~Wen, and S.~Baldi, ``Distributed reinforcement learning
  algorithm for dynamic economic dispatch with unknown generation cost
  functions,'' \emph{IEEE Transactions on Industrial Informatics}, vol.~16,
  no.~4, pp. 2258--2267, 2019.

\bibitem{nedic2009distributed}
A.~Nedic and A.~Ozdaglar, ``Distributed subgradient methods for multi-agent
  optimization,'' \emph{IEEE Transactions on Automatic Control}, vol.~54,
  no.~1, pp. 48--61, 2009.

\bibitem{lou2017privacy}
Y.~Lou, L.~Yu, S.~Wang, and P.~Yi, ``Privacy preservation in distributed
  subgradient optimization algorithms,'' \emph{IEEE transactions on
  cybernetics}, vol.~48, no.~7, pp. 2154--2165, 2017.

\bibitem{wang2010control}
J.~Wang and N.~Elia, ``Control approach to distributed optimization,'' in
  \emph{2010 48th Annual Allerton Conference on Communication, Control, and
  Computing (Allerton)}.\hskip 1em plus 0.5em minus 0.4em\relax IEEE, 2010, pp.
  557--561.

\bibitem{shi2014linear}
W.~Shi, Q.~Ling, K.~Yuan, G.~Wu, and W.~Yin, ``On the linear convergence of the
  admm in decentralized consensus optimization,'' \emph{IEEE Transactions on
  Signal Processing}, vol.~62, no.~7, pp. 1750--1761, 2014.

\bibitem{liu2019communication}
Y.~Liu, W.~Xu, G.~Wu, Z.~Tian, and Q.~Ling, ``Communication-censored admm for
  decentralized consensus optimization,'' \emph{IEEE Transactions on Signal
  Processing}, vol.~67, no.~10, pp. 2565--2579, 2019.

\bibitem{zeng2016distributed}
X.~Zeng, P.~Yi, and Y.~Hong, ``Distributed continuous-time algorithm for
  constrained convex optimizations via nonsmooth analysis approach,''
  \emph{IEEE Transactions on Automatic Control}, vol.~62, no.~10, pp.
  5227--5233, 2016.

\bibitem{liu2017convergence}
S.~Liu, Z.~Qiu, and L.~Xie, ``Convergence rate analysis of distributed
  optimization with projected subgradient algorithm,'' \emph{Automatica},
  vol.~83, pp. 162--169, 2017.

\bibitem{li2021exponentially}
W.~Li, X.~Zeng, S.~Liang, and Y.~Hong, ``Exponentially convergent algorithm
  design for constrained distributed optimization via nonsmooth approach,''
  \emph{IEEE Transactions on Automatic Control}, vol.~67, no.~2, pp. 934--940,
  2021.

\bibitem{liu2015second}
Q.~Liu and J.~Wang, ``A second-order multi-agent network for bound-constrained
  distributed optimization,'' \emph{IEEE Transactions on Automatic Control},
  vol.~60, no.~12, pp. 3310--3315, 2015.

\bibitem{zhang2015distributed}
Y.~Zhang and Y.~Hong, ``Distributed optimization design for high-order
  multi-agent systems,'' in \emph{2015 34th Chinese Control Conference
  (CCC)}.\hskip 1em plus 0.5em minus 0.4em\relax IEEE, 2015, pp. 7251--7256.

\bibitem{zhao2017distributed}
Y.~Zhao, Y.~Liu, G.~Wen, and G.~Chen, ``Distributed optimization for linear
  multiagent systems: Edge-and node-based adaptive designs,'' \emph{IEEE
  Transactions on Automatic Control}, vol.~62, no.~7, pp. 3602--3609, 2017.

\bibitem{taaang2018optimal}
Y.~Tang, Z.~Deng, and Y.~Hong, ``Optimal output consensus of high-order
  multiagent systems with embedded technique,'' \emph{IEEE Transactions on
  Cybernetics}, vol.~49, no.~5, pp. 1768--1779, 2018.

\bibitem{li2019distributed}
Z.~Li, Z.~Wu, Z.~Li, and Z.~Ding, ``Distributed optimal coordination for
  heterogeneous linear multiagent systems with event-triggered mechanisms,''
  \emph{IEEE Transactions on Automatic Control}, vol.~65, no.~4, pp.
  1763--1770, 2019.

\bibitem{yu2021new}
H.~Yu and T.~Chen, ``A new zeno-free event-triggered scheme for robust
  distributed optimal coordination,'' \emph{Automatica}, vol. 129, p. 109639,
  2021.

\bibitem{tang2020optimal}
Y.~Tang and X.~Wang, ``Optimal output consensus for nonlinear multiagent
  systems with both static and dynamic uncertainties,'' \emph{IEEE Transactions
  on Automatic Control}, vol.~66, no.~4, pp. 1733--1740, 2020.

\bibitem{liu2021distributed}
T.~Liu, Z.~Qin, Y.~Hong, and Z.-P. Jiang, ``Distributed optimization of
  nonlinear multiagent systems: a small-gain approach,'' \emph{IEEE
  Transactions on Automatic Control}, vol.~67, no.~2, pp. 676--691, 2021.

\bibitem{yu2019barrier}
X.~Yu, T.~Wang, J.~Qiu, and H.~Gao, ``Barrier lyapunov function-based adaptive
  fault-tolerant control for a class of strict-feedback stochastic nonlinear
  systems,'' \emph{IEEE Transactions on Cybernetics}, vol.~51, no.~2, pp.
  938--946, 2019.

\bibitem{tee2009barrier}
K.~P. Tee, S.~S. Ge, and E.~H. Tay, ``Barrier lyapunov functions for the
  control of output-constrained nonlinear systems,'' \emph{Automatica},
  vol.~45, no.~4, pp. 918--927, 2009.

\bibitem{ames2019control}
A.~D. Ames, S.~Coogan, M.~Egerstedt, G.~Notomista, K.~Sreenath, and P.~Tabuada,
  ``Control barrier functions: Theory and applications,'' in \emph{2019 18th
  European control conference (ECC)}.\hskip 1em plus 0.5em minus 0.4em\relax
  IEEE, 2019, pp. 3420--3431.

\bibitem{ames2016control}
A.~D. Ames, X.~Xu, J.~W. Grizzle, and P.~Tabuada, ``Control barrier function
  based quadratic programs for safety critical systems,'' \emph{IEEE
  Transactions on Automatic Control}, vol.~62, no.~8, pp. 3861--3876, 2016.

\bibitem{xu2015robustness}
X.~Xu, P.~Tabuada, J.~W. Grizzle, and A.~D. Ames, ``Robustness of control
  barrier functions for safety critical control,'' \emph{IFAC-PapersOnLine},
  vol.~48, no.~27, pp. 54--61, 2015.

\bibitem{liang2017distributed}
S.~Liang, X.~Zeng, and Y.~Hong, ``Distributed nonsmooth optimization with
  coupled inequality constraints via modified lagrangian function,'' \emph{IEEE
  Transactions on Automatic Control}, vol.~63, no.~6, pp. 1753--1759, 2017.

\bibitem{xu2018constrained}
X.~Xu, ``Constrained control of input--output linearizable systems using
  control sharing barrier functions,'' \emph{Automatica}, vol.~87, pp.
  195--201, 2018.

\bibitem{wu2021distributed}
C.~Wu, H.~Fang, Q.~Yang, X.~Zeng, Y.~Wei, and J.~Chen, ``Distributed
  cooperative control of redundant mobile manipulators with safety
  constraints,'' \emph{IEEE Transactions on Cybernetics}, 2021.

\bibitem{an2021distributed}
L.~An and G.-H. Yang, ``Distributed optimal coordination for heterogeneous
  linear multiagent systems,'' \emph{IEEE Transactions on Automatic Control},
  vol.~67, no.~12, pp. 6850--6857, 2021.

\bibitem{huang2004nonlinear}
J.~Huang, \emph{Nonlinear Output Regulation: Theory and Applications}.\hskip
  1em plus 0.5em minus 0.4em\relax SIAM, 2004.

\bibitem{jiang1994small}
Z.~P. Jiang, A.~R. Teel, and L.~Praly, ``Small-gain theorem for iss systems and
  applications,'' \emph{Mathematics of Control, Signals and Systems}, vol.~7,
  pp. 95--120, 1994.

\end{thebibliography}

\end{document}